\documentclass{amsart}
\usepackage{amssymb}
\usepackage{graphicx}
\usepackage{xcolor} 

\definecolor{green}{rgb}{0,0.8,0} 

\newtheorem{theorem}{Theorem}[section]

\newtheorem{lemma}[theorem]{Lemma}

\theoremstyle{definition}
\newtheorem{definition}[theorem]{Definition}

\newtheorem*{remark}{Remark}
\numberwithin{equation}{section}
\newcommand{\nrm}[1]{\Vert#1\Vert}
\newcommand{\abs}[1]{\vert#1\vert}
\newcommand{\brk}[1]{\langle#1\rangle}
\newcommand{\set}[1]{\{#1\}}

\renewcommand{\Im}{\mathrm{Im}}

\newcommand{\aleq}{\lesssim}

\newcommand{\lap}{\triangle}

\newcommand{\ud}{\mathrm{d}}
\newcommand{\rd}{\partial}
\newcommand{\nb}{\nabla}

\newcommand{\bb}{\Big}

\newcommand{\alp}{\alpha}
\newcommand{\bt}{\beta}
\newcommand{\gmm}{\gamma}

\newcommand{\dlt}{\delta}

\newcommand{\eps}{\epsilon}

\newcommand{\kpp}{\kappa}
\newcommand{\lmb}{\lambda}

\newcommand{\sgm}{\sigma}

\newcommand{\bfA}{{\bf A}}

\newcommand{\bfD}{{\bf D}}

\newcommand{\bfI}{{\bf I}}

\newcommand{\bbC}{\mathbb C}

\newcommand{\bbR}{\mathbb R}

\newcommand{\calB}{\mathcal B}

\newcommand{\calD}{\mathcal D}
\newcommand{\calE}{\mathcal E}
\newcommand{\calF}{\mathcal F}

\newcommand{\calL}{\mathcal L}

\newcommand{\calS}{\mathcal S}

\newcommand{\Fleq}{\preceq}				
\newcommand{\Fgeq}{\succeq}				
\newcommand{\sptF}[1]{\calF \left[ #1 \right]} 	
\newcommand{\Id}{\bfI_{2 \times 2}}			
\newcommand{\Dirac}{\calD}				
\newcommand{\covD}{\bfD}				
\newcommand{\cf}{\mathrm{cf}}			
\newcommand{\df}{\mathrm{df}}			
\newcommand{\mR}{\mathfrak{R}}			
\newcommand{\err}{\calE}			
\newcommand{\phihom}{\phi^{\mathrm{hom}}}	
\newcommand{\Ahom}{A^{\mathrm{hom}}}		
\newcommand{\psihom}{\psi^{\mathrm{hom}}}	
\newcommand{\nf}{\mathcal{Q}}			
\newcommand{\vect}[2]{\left( \begin{array}{c} #1 \\ #2 \end{array} \right)}				
\newcommand{\matr}[4]{\left( \begin{array}{cc} #1 & #2 \\ #3 & #4 \end{array} \right)}		
\newcommand{\curl}{\mathrm{curl}\,}			
\renewcommand{\div}{\mathrm{div}\,}			

\begin{document}

\title[]{Low regularity solutions to the Chern-Simons-Dirac and the Chern-Simons-Higgs equations in the Lorenz gauge}
\author{Hyungjin Huh}%
\address{Department of Mathematics, Chung-Ang University, Seoul, 156-756, Republic of Korea}%
\email{huh@cau.ac.kr}%

\author{Sung-Jin Oh}%
\address{Department of Mathematics, Princeton University, Princeton, New Jersey, 08544}%
\email{sungoh@math.princeton.edu}%


\begin{abstract}
In this paper, we address the problem of local well-posedness of the Chern-Simons-Dirac (CSD) and the Chern-Simons-Higgs (CSH) equations in the Lorenz gauge for low regularity initial data.  One of our main contributions is the uncovering of a null structure of (CSD). Combined with the standard machinery of $X^{s,b}$ spaces, we obtain local well-posedness of (CSD) for initial data $a_{\mu}, \psi \in H^{1/4+\eps}_{x}$. 
Moreover, it is observed that the same techniques applied to (CSH) lead to a quick proof of local-wellposedness for initial data $a_{\mu} \in H^{1/4+\eps}_{x}$, $(\phi, \rd_{t} \phi) \in H^{3/4+\eps}_{x} \times H^{-1/4 + \eps}_{x}$, which improves the previous result of \cite{Selberg:2012vb}.
\end{abstract}
\maketitle

\section{Introduction}

In this paper, we continue to study the initial value problem for the  \emph{Chern-Simons-Dirac} \eqref{eq:CSD:original}  equations
 \begin{equation} \label{eq:CSD:original} \tag{CSD}
 \left\{
\begin{aligned}
	& F_{\mu \nu} = - \frac{2}{\kpp} \eps_{\mu \nu \lmb} (\overline{\psi} \gmm^{\lmb} \psi), \\
	& i \gmm^{\mu} \covD_{\mu} \psi - m \psi = 0,
\end{aligned}
\right.
\end{equation}
and the \emph{Chern-Simons-Higgs} \eqref{eq:CSH:original} equations
\begin{equation} \label{eq:CSH:original} \tag{CSH}
	\left\{
	\begin{aligned}
	& F_{\mu \nu} = \frac{2}{\kpp} \eps_{\mu \nu \lmb} \Im (\overline{\phi} \covD^{\lmb} \phi), \\
	& \covD^{\mu} \covD_{\mu } \phi = 0,
	\end{aligned}
	\right.
\end{equation}
on the Minkowski space $\bbR^{1+2}$. We remark that, for the sake of simplicity,  the potential term $\phi V'(\abs{\phi}^{2})$ on the right-hand side of the Higgs equation\footnote{ We remark that a polynomial-type potential $V(r)$, with an appropriate restriction on the degree, can also be incorporated in our main theorem for \eqref{eq:CSH:original} (Theorem \ref{thm:lwp4CSH}), using the techniques of \cite{Selberg:2012vb}.} is omitted. 

Here, the 1-form $A_{\mu}$ is to be interpreted as defining a connection; accordingly, $F := \ud A$ is the curvature 2-form, and $\covD_{\mu}:=\rd_{\mu} - i A_{\mu}$ is the associated covariant derivative. The spinor field $\psi$ is represented by a column vector with 2 complex components. The notation $\psi^{\dag}$ refers to the complex conjugate transpose of $\psi$, whereas $\overline{\psi}$ is defined by $\overline{\psi}=\psi^{\dag}\gamma^0$. For the definition of the gamma matrices $\gmm_{\mu}$, we refer the reader to \S \ref{subsec:prelim:dirac}. The totally skew-symmetric tensor $\eps_{\mu \nu \lmb}$ is characterized by $\epsilon_{012}=1$, $\kappa>0$ is the Chern-Simons coupling constant and the nonnegative constant $m$ is the mass of the spinor field $\psi$. The scalar field $\phi$ is a complex-valued function. 

We will work on the Minkowski space $\bbR^{1+2}$ equipped with the Minkowski metric of signature $(+ - -)$. Greek indices, such as $\mu, \nu$, will refer to all indices $0,1,2$, whereas latin indices, such as $i, j, k, \ell$, will refer only to the spatial indices $1,2$, unless otherwise specified. We will use the Minkowski metric to raise and lower the indices. Moreover, we will adopt the Einstein summation convention of summing up repeated upper and lower indices.

The \eqref{eq:CSD:original} and \eqref{eq:CSH:original} systems  were introduced in \cite{CKP, LB} and  \cite{HKP, JW} respectively to deal with the
electromagnetic phenomena in planar domains such as fractional quantum Hall effect or high temperature super conductivity. We refer to \cite{Dun, H} for more physical information.

The \eqref{eq:CSD:original} system exhibits the \emph{conservation of charge}\footnote{We remark that \eqref{eq:CSD:original} also possesses a conserved energy, but the energy density is not positive, and hence difficult to utilize.}
\begin{equation}\label{eq:CSD:char}
Q(t)= \int_{\Bbb{R}^2} | \psi(t,x) |^2 dx=Q(0),
\end{equation}
whereas the  \eqref{eq:CSH:original} system obeys the \emph{conservation of the total energy}
\begin{equation}\label{eq:CSH:ener}
E(t)= \sum_{\mu=0,1,2} \frac{1}{2} \int_{\Bbb{R}^2} \abs{\covD_{\mu}\phi(t,x)}^2  \,dx =E(0).
\end{equation}

In order to figure out the optimal regularity for the \eqref{eq:CSD:original} system, observe that in the case of zero mass $m=0$, this system is invariant under the scaling
\begin{align*}
  \psi^{(\lambda)}(t,\,x) =\lambda\psi(\lambda t,\,\lambda x),\quad 
A^{(\lambda)}_{\mu}(t,\,x) =\lambda A_{\mu}(\lambda t,\,\lambda x).
\end{align*}
From this, we can derive the scaling critical Sobolev exponent $s_{c} = 0$ for $\psi$ and $A_{\mu}$. In view of \eqref{eq:CSD:char} we can say that the initial-value problem of \eqref{eq:CSD:original} is \emph{charge critical}. 
On the other hand, the \eqref{eq:CSH:original} equations is invariant under 
\begin{align*}
  \phi^{(\lambda)}(t,\,x) =\lambda^{1/2}\phi(\lambda t,\,\lambda x),\quad 
A^{(\lambda)}_{\mu}(t,\,x) =\lambda A_{\mu}(\lambda t,\,\lambda x).
\end{align*}
Therefore, the scaling critical Sobolev exponents are $s_c=1/2$ for $\phi$ and $s_c=0$ for $A_{\mu}$. 
In comparison with \eqref{eq:CSH:ener}, the Cauchy problem of \eqref{eq:CSH:original} is said to be \emph{energy sub-critical}. 

Another important property of the systems \eqref{eq:CSD:original} and \eqref{eq:CSH:original} is the gauge invariance. As such, a solution to the equations \eqref{eq:CSD:original} and \eqref{eq:CSH:original} is formed by a class of gauge equivalent pairs $(A_{\mu},\psi)$ and $(A_{\mu},\phi)$, respectively. In this work, we will fix the gauge by imposing the \emph{Lorenz gauge} condition $\rd^{\mu} A_{\mu} = 0$.

\subsection{Previous work}
The Cauchy problem for \eqref{eq:CSD:original} has been studied in \cite{Huh-d} by one of the authors (H. Huh) in the Lorenz, Coulomb and temporal gauges. Depending on the gauge choice, the \eqref{eq:CSD:original} system showed different features which were crucial in the study. Under the Lorenz gauge condition, in particular, local well-posedness of \eqref{eq:CSD:original} was established for initial data with regularity 
\begin{align*}
 A_{\mu}(0,x) \in H^{1/2}_x, \qquad \psi(0,x) \in H^{5/8}_x.
 \end{align*}
It is of interest to lower the regularity condition on the spinor field\footnote{The regularity required for the connection 1-form $A_{\mu}$, on the other hand, may be viewed as having less importance. The reason is because the \eqref{eq:CSD:original} system allows one to control the curvature $F_{\mu \nu}$ in terms of $\psi$, from which $A_{\mu}$ can be recovered under a suitable gauge choice.} $\psi(0,x)$ to $L^{2}_{x}$, in view of the fact that the conserved charge of \eqref{eq:CSD:original} is exactly the $L^{2}_{x}$ norm of $\psi(0,x)$. In this aspect, the best result so far, also proved in \cite{Huh-d}, is the local well-posedness result for initial data satisfying
\begin{equation*}
	A_{\mu}(0,x) \in L^{2}_{x}, \qquad \psi(0,x) \in H^{1/2+\eps}_{x},
\end{equation*}
for arbitrarily small $\eps > 0$, which was proved under the Coulomb gauge condition $\rd^{\ell} A_{\ell} = 0$.

On the other hand, the Cauchy problem for \eqref{eq:CSH:original} has been investigated in  \cite{CC, Huh-07, Huh-11, Selberg:2012vb}.  Global existence of solutions to \eqref{eq:CSH:original} with initial data $A_{\mu}(0, x)\in H^s_{x} $, $\phi(0, x)\in H^{s+1}_{x} $, $\partial_t\phi(0, x)\in H^s_{x}$ for $s\geq 1$ was shown by D. Chae and K. Choe \cite{CC}, under the Coulomb gauge condition. The regularity condition on the initial data has been improved by S. Selberg and A. Tesfahun \cite{Selberg:2012vb}, under the Lorenz gauge, to $A_{\mu}(0, x)\in \dot{H}^{1/2}_{x}$, $\phi(0, x)\in H^{1}_{x}$, $\partial_t\phi(0, x)\in L^2_{x}$. Note that this is the \emph{energy regularity}. The authors of \cite{Selberg:2012vb} also constructed a low regularity local-in-time solution satisfying
\begin{equation*}
A_{\mu}\in C_t((-T,T), \, H^{3/8 + \epsilon}_x), \quad \phi\in  C_t((-T,T),\, H^{7/8+\epsilon}_x)\cap C_t^{1} ( (-T, T), \,H^{-1/8+\eps}_{x}),
\end{equation*}
for arbitrarily small $\eps > 0$, which is below the energy regularity. 


\subsection{Main results}
In this paper, we make progress in establishing local well-posedness of both \eqref{eq:CSD:original} and \eqref{eq:CSH:original} systems, in the Lorenz gauge, for initial data with low regularity.

Let us begin with our result on \eqref{eq:CSD:original}. 
In general, an initial data set for this system consists of $(a_{0}, a_{1}, a_{2}, \psi_{0})$, where $A_{\mu}(0,x) = a_{\mu}(x)$ is a real-valued 1-form restricted to $\set{0} \times \bbR^{2}$  and $\psi(0, x) = \psi_{0}(x)$ is a spinor field (i.e. a $\bbC^{2}$-valued function) on $\set{0} \times \bbR^{2}$. Note that the equation for $F_{12} = \rd_{1} A_{2} - \rd_{2} A_{1}$ imposes a constraint for the initial data, which is that it should satisfy
\begin{equation} \label{eq:CSD:constraint}
	\rd_{1} a_{2} - \rd_{2} a_{1} = - \frac{2}{\kpp} (\psi_{0}^{\dagger} \psi_{0}).
\end{equation}
We will say that $(a_{0}, a_{1}, a_{2}, \psi_{0})$ is an \emph{$H^{s}$ initial data set} for \eqref{eq:CSD:original} if \eqref{eq:CSD:constraint} is satisfied and $a_{\mu}, \psi_{0} \in H^{s}_{x}$. Our first main theorem is that the \eqref{eq:CSD:original} system is well-posed for $H^{1/4+\eps}$ initial data in the Lorenz gauge, which is an improvement over the previous results of one of the authors \cite{Huh-d}.

\begin{theorem}[Local well-posedness for $H^{1/4+\eps}$ data] \label{thm:lwp4CSD}
Suppose that $(a_{0}, a_{1}, a_{2}, \psi_{0})$ is an $H^{1/4 + \eps}$ initial data set for the Chern-Simons-Dirac equations, where $0 < \eps \ll 1$. Then there exists $T = T(\nrm{a_{\mu}}_{H^{1/4+\eps}}, \nrm{\psi_{0}}_{H^{1/4+\eps}}, m) > 0$ such that a solution $(A_{\mu}, \psi)$ to \eqref{eq:CSD:original} under the Lorenz gauge condition $\rd^{\mu} A_{\mu} = 0$ with the prescribed initial data exists on the time interval $(-T, T)$, which satisfies
\begin{equation*}
	(A_{\mu}, \psi) \in C_t ( (-T, T), H^{1/4+\eps}_{x}).
\end{equation*}
\end{theorem}

\begin{remark} 
The proof of Theorem \ref{thm:lwp4CSD} will proceed by a Picard iteration argument;  other standard statements of local well-posedness (such as uniqueness in the iteration space, continuous dependence on the data and persistence of regularity) follow easily. We omit the precise statements for the sake of brevity.
\end{remark}

Note that Theorem \ref{thm:lwp4CSD} improves the result of \cite{Huh-d} in the Lorenz gauge, in terms of the regularity condition on both $a_{\mu}$ and $\psi_{0}$. Moreover, even compared to the results of \cite{Huh-d} in other gauges (in particular the Coulomb gauge), the regularity required for the spinor field $\psi_{0}$ is lower.

Key to our proof of Theorem \ref{thm:lwp4CSD} is the uncovering of a null structure of \eqref{eq:CSD:original} in the Lorenz gauge, which is different from that exhibited in \cite{Huh-d}. In particular, the null structure for the wave equations for $A_{\mu}$ seems new. Another novelty of our proof, which we were not able to find in the existing literature, is a simple variant of the Duhamel's formula (Lemma \ref{lem:mDuhamel}) via integration-by-parts which cancels the cumbersome singularity $\abs{\nb}^{-1}$ in the wave kernel, yet retains `structure' so that null structure can be revealed. 

It turns out that our techniques for \eqref{eq:CSD:original} can be easily applied to \eqref{eq:CSH:original} in the Lorenz gauge as well. Let us denote an initial data set for this system by $(a_{\mu}, f, g)$, where $A_{\mu}(0, x) = a_{\mu}$ is a real-valued 1-form restricted to $\set{0} \times \bbR^{2}$ and $(\phi, \rd_{t} \phi)(0, x) = (f,\, g)$ is a pair of complex-valued functions on $\set{0} \times \bbR^{2}$. As in the case of \eqref{eq:CSD:original}, we will say that $(a_{\mu}, f, g)$ is an \emph{$H^{s+1/2}$ initial data set} for \eqref{eq:CSH:original} if $a_{\mu} \in H^{s}_{x}$, $(f, g) \in H^{s+1/2}_{x} \times H^{s-1/2}_{x}$ and the following \emph{constraint equation} is satisfied.
\begin{equation} \label{eq:CSH:constraint}
	\rd_{1} a_{2} - \rd_{2} a_{1} = \frac{2}{\kpp} \Im \big( \overline{f} (g - i a_{0} f) \big ).
\end{equation}

In the last section of this paper, we will establish the following local well-posedness result, which is an improvement over the previous result of Selberg and Tesfahun \cite{Selberg:2012vb} in terms of regularity required for $\phi$ and $A_{\mu}$. This is our second main theorem.

\begin{theorem}[Local well-posedness for $H^{3/4+\eps}$ data] \label{thm:lwp4CSH}
Suppose that $(a_{0}, a_{1}, a_{2}, f, g)$ is an $H^{3/4 + \eps}$ initial data set for the Chern-Simons-Higgs equations, where $0 < \eps \ll 1$. Then there exists $T = T(\nrm{a_{\mu}}_{H^{1/4+\eps}}, \nrm{(f, g)}_{H^{3/4+\eps} \times H^{-1/4+\eps}}) > 0$ such that a solution $(A_{\mu}, \phi)$ to \eqref{eq:CSH:original} under the Lorenz gauge condition $\rd^{\mu} A_{\mu} = 0$ with the prescribed initial data exists on the time interval $(-T, T)$, which satisfies
\begin{equation*}
	A_{\mu} \in C_t ( (-T, T), H^{1/4+\eps}_{x}), \quad \phi \in C_t ( (-T, T), H^{3/4+\eps}_{x}) \cap C_t^{1} ( (-T, T), H^{-1/4+\eps}_{x}).
\end{equation*}
\end{theorem}


\section{Preliminaries}
\subsection{Notations and conventions}
We collect here a few notations and conventions used in this paper.
\begin{itemize}
\item We will use the convention of writing $A \aleq B$ as a shorthand for $\abs{A} \leq C B$ for some $C > 0$.  
\item Given a nice (say Schwartz) function on $\bbR^{1+2}$, we will use the notation $\sptF{\phi}$ to refer to the \emph{space-time Fourier transform} of $\phi$.  Alternatively, we may also write $\widetilde{\phi} = \sptF{\phi}$. 
\item Given two functions $\phi$, $\psi$ on $\bbR^{1+2}$, we will use the notation $\phi \Fleq \psi$ to denote $\abs{\sptF{\phi}} \aleq \sptF{\psi}$.
\end{itemize}

\subsection{Half-wave decomposition of the d'Alembertian} 
Following our sign convention, the d'Alembertian is defined as $\Box = \rd_{t}^{2} - \lap$, where $\lap = \sum_{j} \rd_{j}^{2}$.
In order to make the dispersive properties of the wave equation more clear, it is useful to write the wave equation $\Box \phi = F$ as a first order system
\begin{equation*}
\frac{\rd}{\rd t} \vect{\phi}{\phi_{t}} = \matr{0}{1}{\lap}{0} \vect{\phi}{\phi_{t}} + \vect{0}{F},
\end{equation*}
and diagonalize the symbol, by conjugating with the Fourier transform. We are then led to the transform $(\phi, \phi_{t}) \to (\phi_{+}, \phi_{-})$ and $(0, F) \to (F_{+}, F_{-})$, where
\begin{equation*}
	\phi_{\pm} = \frac{1}{2} (\phi \mp \frac{1}{i \abs{\nb}} \phi_{t}), \hbox{ and } F_{\pm} = \mp \frac{1}{2i \abs{\nb}} F,
\end{equation*}
with $\abs{\nb} = \sqrt{- \lap}$ as usual. We then obtain the following diagonal first order system
\begin{equation*}
\frac{\rd}{\rd t} \vect{\phi_{+}}{\phi_{-}} = \matr{-i\abs{\nb}}{0}{0}{+i\abs{\nb}} \vect{\phi_{+}}{\phi_{-}} + \vect{F_{+}}{F_{-}},
\end{equation*}
or equivalently, the following pair of \emph{half-wave equations}
\begin{equation} \label{}
	(-i\rd_{t}  \pm  \abs{\nb}) \phi_{\pm} = \pm \frac{1}{2 \abs{\nb}} F.
\end{equation}


For the purpose of revealing null structures, it is useful to introduce the \emph{modified Riesz transforms} $\mR^{\mu}$, which are self-adjoint operators defined as follows.
\begin{align*}
&\mR^{0}_{\pm} = \mR_{\pm,0}   := -1, \\
& \mR^{j}_{\pm} = -\mR_{\pm,j}  := \mp \bb( \frac{\rd_{j}}{i \abs{\nb}} \bb).
\end{align*}
The modified Riesz transform $\mR_{\pm, \mu}$ models the usual Riesz transform applied to a free $(\pm)$-half-wave (up to a sign). Indeed, let $\phihom_{\pm}$ be a solution to $(-i \rd_{0} \pm \abs{\nb}) \phihom_{\pm} = 0$. Then we have
\begin{equation*}
\frac{\rd_{\mu}}{i \abs{\nb}} \phihom_{\pm} = \pm \, \mR_{\pm, \mu} \phihom_{\pm}.
\end{equation*}

The following lemma is a variant of the usual Duhamel's formula when $\Box \phi$ is of a divergence form; it is essentially an integration by parts in time.
\begin{lemma} \label{lem:mDuhamel}
Let $\phi$, $F_{\mu}$ ($\mu = 0,1,2$) be Schwartz functions, such that\footnote{We remark that the Einstein summation convention of summing up repeated upper and lower indices is in order.}
\begin{equation*}
	\Box \phi = \rd^{\mu} F_{\mu}.
\end{equation*}
Then the half-waves $\phi_{\pm}$ are given by the Duhamel's formula
\begin{equation*}
	\phi_{\pm} (t,x) = \phihom_{\pm} (t,x) - \frac{1}{2} \int_{0}^{t} e^{\pm (-i) (t-s) \abs{\nb}}  \mR_{\pm}^{\mu} F_{\mu}(s, x)  \, \ud s,
\end{equation*}
where $\phihom_{\pm}(t,x)$ are the free half-waves given by
\begin{equation*}
\phihom_{\pm}(t,x) := \frac{1}{2} e^{\pm (-i) t \abs{\nb}} \bb( \phi (0, x) \pm \frac{1}{i \abs{\nb}} (F_{0}(0,x) - \rd_{0} \phi(0,x) )\bb).
\end{equation*}
\end{lemma}
\begin{proof} 
Let us start with the usual Duhamel's formula
\begin{equation*}
	\phi(t,x) = \phihom(t,x) + \int_{0}^{t} \frac{\sin (t-s) \abs{\nb}}{\abs{\nb}} \rd_{0} F_{0}(s,x) \, \ud s + \int_{0}^{t} \frac{\sin (t-s) \abs{\nb}}{\abs{\nb}} \rd^{j} F_{j}(s,x) \, \ud s.
\end{equation*}
where
\begin{equation*}
\phihom(t,x) = \frac{1}{2} \sum_{\pm} e^{\pm (-i) t \abs{\nb}} \bb( \phi(0, x) \mp \frac{1}{i \abs{\nb}} \rd_{0} \phi(0,x) \bb).
\end{equation*}
Expanding out $\sin(t-s) \abs{\nb}$ of the last term in terms of exponentials, we easily see that
\begin{equation*}
\int_{0}^{t} \frac{\sin (t-s) \abs{\nb}}{\abs{\nb}} \rd^{j} F_{j}(s,x) \, \ud s 
= - \sum_{\pm} \frac{1}{2} \int_{0}^{t} e^{\pm (-i) (t-s)\abs{\nb}} \mR^{j}_{\pm} F_{j} (s,x) \, \ud s.
\end{equation*}
On the other hand, for the second term we proceed as follows.
\begin{align*}
	&\int_{0}^{t} \frac{\sin (t-s) \abs{\nb}}{\abs{\nb}} \rd_{0} F_{0}(s,x) \, \ud s  \\
	& \quad = \int_{0}^{t} \cos((t-s) \abs{\nb}) F_{0}(s,x) \, \ud s 
	- \frac{\sin t \abs{\nb}}{\abs{\nb}} F_{0}(0, x)\\
	& \quad = \sum_{\pm} \bb( \pm e^{\pm(- i) t \abs{\nb}} \frac{1}{2i \abs{\nb}}  F_{0}(0, x)
		- \frac{1}{2} \int_{0}^{t} e^{\pm (-i) (t-s) \abs{\nb}} \mR^{0}_{\pm} F_{0}(s,x) \, \ud s \bb),
\end{align*}
from which the lemma follows. We remind the reader that $\rd^{0} =\rd_{0} = \rd_{t}$. The first equality may be justified rigorously by, say, using the Fourier transform.
\end{proof}

This lemma is useful because the operator $\abs{\nb}^{-1}$ (which is unfavorable for low frequencies, especially on $\bbR^{2}$) in the wave kernel is canceled, whereas $\mR^{\mu}_{\pm}$ remains; the latter is important for revealing a null structure after a duality argument.

We will also consider a massive Klein-Gordon equation
\begin{equation} \label{eq:KG}
	(\Box + 1) \phi = F.
\end{equation}
For this equation, the $(\pm)$-half-wave decomposition takes the form
\begin{equation} \label{eq:hwDecomp4KG}
\phi_{\pm} = \frac{1}{2} (\phi \mp \frac{1}{i \brk{\nb}} \phi_{t}),
\end{equation}
where $\brk{\nb} := (1-\lap)^{1/2}$, and \eqref{eq:KG} becomes
\begin{equation} \label{eq:KG:hw}
	(-i \rd_{t} \pm \brk{\nb}) \phi_{\pm} = \pm \frac{1}{2 \brk{\nb}} F.
\end{equation}

\subsection{Dirac operator} \label{subsec:prelim:dirac}
Let $\eta_{\mu \nu}$ be the Minkowski metric on $\bbR^{1+2}$ with signature $(+, -, -)$. Consider the \emph{gamma matrices} $\gmm^{\mu}$ ($\mu = 0,1,2$) which are defined as follows:
\begin{equation*}
	\gmm^{0} = \sgm^{3}, \quad
	\gmm^{1} = i \sgm^{2}, \quad
	\gmm^{2} = -i \sgm^{1}.
\end{equation*}
Here, $\sgm^{j}$ $(j=1,2,3)$ are the \emph{Pauli matrices}, which are $2 \times 2$-matrices of the form
\begin{equation*}
	\sgm^{1} = \left( \begin{array}{cc} 0 & 1 \\ 1 & 0 \end{array} \right), \quad
	\sgm^{2} = \left( \begin{array}{cc} 0 & -i \\ i & 0 \end{array} \right), \quad
	\sgm^{3} = \left( \begin{array}{cc} 1 & 0 \\ 0 & -1 \end{array} \right).
\end{equation*}
Note that $\sgm^{j}$ satisfy the following algebraic properties (for $j, k = 1,2,3$):
\begin{equation*}
	\frac{1}{2}(\sgm^{j} \sgm^{k} + \sgm^{k} \sgm^{j}) = \dlt^{jk} \Id, \quad \sgm^{1} \sgm^{2} \sgm^{3} = i \Id.
\end{equation*}
Furthermore, $\gmm^{\mu}$ obey the following multiplication law:
\begin{equation*}
	\frac{1}{2} (\gmm^{\mu} \gmm^{\nu} +\gmm^{\mu} \gmm^{\nu}) = \eta^{\mu \nu} \Id.
\end{equation*}


Let $\psi$ be a \emph{2-spinor field}, i.e. $\bbC^{2}$-valued function on $\bbR^{1+2}$. 
The \emph{massive Dirac operator} of mass $m$ $(m \geq 0)$ is defined by
\begin{equation*}
	\Dirac \psi := (i \gmm^{\mu} \rd_{\mu} - m) \psi.
\end{equation*}
The equation $\Dirac \psi = 0$ is called the \emph{Dirac equation}. The Dirac equation is a Lagrangian field theory; i.e.  $\calD \psi = 0$ is the Euler-Lagrange equation of an action functional given by a Lagrangian, which we shall describe below. Given a spinor $\psi$, we define its \emph{Dirac adjoint} to be $\overline{\psi} = \psi^{\dagger} \gmm^{0}$. The \emph{Dirac Lagrangian} is defined to be
\begin{equation*}
	\calL_{D} [\psi] = i \overline{\psi} \, \gmm^{\mu} \rd_{\mu} \psi - m \overline{\psi} \psi.
\end{equation*}

In what follows, we will outline the approach to study the Dirac equation taken in \cite{DAncona:2007ex}. Since $\gmm^{0}$  acts differently on components, in practice it is often convenient to use the original $\bt, \,\alp^{i}$ formulation of the Dirac operator, which goes back to P. M. Dirac. For this purpose, we multiply the Dirac operator 
by $\gmm^{0}$ on the left and define
\begin{equation*}
	\bt := \gmm^{0}, \quad \alp^{1} := \gmm^{0} \gmm^{1} =  \sgm^{1}, \quad \alp^{2} := \gmm^{0} \gmm^{2} = \sgm^{2}.
\end{equation*}
For convenience, we also define $\alp^{0} := \Id$.  Then the following  relations hold:
\begin{equation*}
	\bt \alp^{i} + \alp^{i} \bt = 0, \quad \frac{1}{2}(\alp^{i} \alp^{j} + \alp^{j} \alp^{i}) = \dlt^{ij} \Id.
\end{equation*}
We then have
\begin{equation*}
	\bt \calD \psi = i \rd_{0} \psi + i \alp^{j} \rd_{j} \psi - m \bt \psi.
\end{equation*}

In order to study this first order PDE, it is natural to rewrite the above equation in the eigenbasis of the spatial operator $-i \alp^{j} \rd_{j}$. Using the Fourier transform, this operator is conjugated to $\xi_{j}\alp^{j} $. Since $(\xi_{j}\alp^{j})^{2} = \abs{\xi}^{2} \Id$, the only possible eigenvalues of $\xi_{j}\alp^{j}$ are $\pm \abs{\xi}$; by symmetry, we can easily infer that each eigenvalue corresponds to one eigenvector. Now consider the $2\times2$-matrrix $\Pi = \Pi(\xi)$ defined by
\begin{equation*}
	\Pi(\xi) := \frac{1}{2} \bb( \Id + \frac{\xi_{j}\alp^{j}}{\abs{\xi}} \bb),
\end{equation*}
and define $\Pi_{\pm} (\xi) := \Pi(\pm \xi)$. It is easy to verify that $\Pi_{\pm}$ is the projection to the eigenspace corresponding to $\pm \abs{\xi}$, respectively. Equivalently, the following algebraic relations hold.
\begin{align*}
	& \Pi_{\pm}(\xi)^{2} = \Pi_{\pm}(\xi), \quad \Pi_{+}(\xi) \Pi_{-}(\xi) = 0, \\
	& \Id = \Pi_{+}(\xi) + \Pi_{-}(\xi), \quad \xi_{j}\alp^{j} = \abs{\xi} \Pi_{+}(\xi) - \abs{\xi} \Pi_{-}(\xi).
\end{align*}
Note, furthermore, that $\Pi$ satisfy the following laws.
\begin{align}
& \Pi_{\pm}(\xi) = \Pi_{\mp}(-\xi) = \Pi(\pm \xi), \quad \bt \Pi(\xi) = \Pi(-\xi) \bt , \\
& \alp^{i} \Pi(\xi) = \Pi(-\xi) \alp^{i}  + \frac{\xi_{i}}{\abs{\xi}} \Id. \label{eq:dirac:1}
\end{align}

We will use the notation $\Pi_{\pm} := \Pi_{\pm}(\nb/i)$, and define $\psi_{\pm} := \Pi_{\pm} \psi$. Applying $\Pi_{\pm}$, the Dirac operator becomes
\begin{equation*}
	\Pi_{\pm} (\bt \Dirac \psi) = -(-i \rd_{0} \pm \abs{\nb}) \psi_{\pm} - m \bt \psi_{\mp},
\end{equation*}
whose principal term is nothing but the \emph{half-wave} operator. In particular, when $m=0$ and $\calD \psi = 0$, then $\psi_{\pm}$ is a free $(\pm)$-half-wave. This will be the basis of the Fourier analytic study of the Dirac operator. 

Utilizing the modified Riesz transforms $\mR_{\pm}^{\mu}$,  the quantization $\xi^{j} \to- i \rd^{j}$ of the relation \eqref{eq:dirac:1} may be rewritten concisely as
\begin{equation} \label{eq:commAlpPi}
	\alp^{\mu} \Pi_{\pm} =  \Pi_{\mp} \alp^{\mu} \Pi_{\pm} - \mR^{\mu}_{\pm} \Pi_{\pm},
\end{equation}
where we remark that the case $\mu = 0$ is rather trivial.

\subsection{$X^{s,b}_{\pm}$ and $H^{s,b}$ spaces} \label{subsec:prelim:Xsb}
The key technical tool that we shall use to prove low regularity local well-posedness is the $X^{s,b}$-type spaces, which was introduced by Klainerman-Machedon \cite{Klainerman:1995vs} in the context of nonlinear wave equations and Bourgain \cite{MR1209299}, \cite{MR1215780} for NLS and KdV. In this subsection, we briefly summarize the portion of the theory that will be used in this paper.

Let $\phi$ be a Schwartz function on $\bbR^{1+2}$. For $s, b \in \bbR$ and a sign $\pm$, we define the $X^{s,b}_{\pm}$ norm of $\phi$ by
\begin{equation*}
	\nrm{\phi}_{X^{s,b}_{\pm}} := \nrm{(1 + \abs{\xi})^{s} (1 + \abs{\tau \pm \abs{\xi}})^{b} \widetilde{\phi}(\tau, \xi)}_{L^{2}_{\tau, \xi}}. 
\end{equation*}
We also define the $H^{s,b}$ norm as follows.
\begin{equation*}
	\nrm{\phi}_{H^{s,b}} := \nrm{(1 + \abs{\xi})^{s} (1 + \abs{\abs{\tau} - \abs{\xi}})^{b} \widetilde{\phi}(\tau, \xi)}_{L^{2}_{\tau, \xi}}. 
\end{equation*}
The Banach spaces $X^{s,b}_{\pm}$ and $H^{s,b}$ are defined by taking the closure of $\calS(\bbR^{1+2})$ with respect to each norm. For an arbitrary sign $\pm$ and $s, b \in \bbR$, note that the norms $X^{s,b}_{\pm}, H^{s,b}$ of $\phi \in \calS(\bbR^{1+2})$ depend only on the size of its space-time Fourier transform; i.e. for $\phi \Fleq \psi$, we have
\begin{equation*}
	\nrm{\phi}_{X^{s,b}_{\pm}} \aleq \nrm{\psi}_{X^{s,b}_{\pm}}, \quad \nrm{\phi}_{H^{s,b}} \aleq \nrm{\psi}_{H^{s,b}},
\end{equation*}
where we remind the reader that $\phi \Fleq \psi$ is the shorthand for $\abs{\widetilde{\phi}} \aleq \widetilde{\psi}$. Moreover, for $b \geq 0$, observe that the following inclusion relations hold.
\begin{equation*}
	X^{s,b}_{\pm} \subset H^{s, b} \subset  H^{s, -b} \subset X^{s,-b}_{\pm}.
\end{equation*}
Note that the spaces $X^{s,b}_{\pm}$ are, by nature, defined globally on $\bbR^{1+2}$. In order to utilize these spaces in the local-in-time setting, we will introduce the notion of the \emph{restriction space}. Given $T > 0$, consider the subset $S_{T} := (-T, T) \times \bbR^{2}$ of $\bbR^{1+2}$. We define the \emph{restriction norm} $X^{s,b}_{\pm}(S_{T})$ for a function $\phi$ on $S_{T}$ by
\begin{equation*}
	\nrm{\phi}_{X^{s,b}_{\pm}(S_{T})} := \inf \set{\nrm{\psi}_{X^{s,b}_{\pm}} : \psi \in X^{s,b}_{\pm}, \psi = \phi \hbox{ on } S_{T}}.
\end{equation*}

Let us state a few lemmas regarding $X^{s,b}_{\pm}$ spaces; for proofs, we refer the reader to \cite[Section 4]{DAncona:2007ex} or \cite[\S 2.6]{MR2233925}.

\begin{lemma}[Embedding into $C_{t} H^{s}_{x}$] \label{lem:Xsb2Hs}
For $T > 0$ and $b > 1/2$, the following embedding holds.
\begin{equation*}
X^{s,b}_{\pm}(S_{T}) \subset C_{t} ((-T, T), H^{s}_{x}).
\end{equation*}
\end{lemma}

\begin{lemma}[Estimate for homogeneous waves] \label{lem:XsbHom}
For every $T > 0$, $s \in \bbR$, $b > 1/2$ and $\phi_{0} \in H^{s}_{x}$, the following estimate holds.
\begin{equation*}
	\nrm{e^{\pm(-i) t \abs{\nb}} \phi_{0}}_{X^{s,b}_{\pm}(S_{T})} \aleq \nrm{\phi_{0}}_{H^{s}_{x}}, \quad
	\nrm{e^{\pm(-i) t \brk{\nb}} \phi_{0}}_{X^{s,b}_{\pm}(S_{T})} \aleq \nrm{\phi_{0}}_{H^{s}_{x}}.
\end{equation*}
\end{lemma}

\begin{lemma}[Estimate for inhomogeneous waves] \label{lem:XsbInhom}
Let $-1/2 < b < b'$, and consider an inhomogeneous half-wave equation $(-i \rd_{t} \pm \abs{\nb}) \phi_{\pm} = F_{\pm}$ with zero data at $t=0$. If $F \in X^{s,b'-1}_{\pm}$, then there exists a unique solution $\phi_{\pm} \in X^{s,b}_{\pm}$ to this equations such that for every $T > 0$ the following estimate holds.
\begin{equation*}
	\nrm{\phi_{\pm}}_{X^{s,b}_{\pm} (S_{T})} \aleq T^{b' - b} \nrm{F_{\pm}}_{X^{s, b'-1}_{\pm}(S_{T})}.
\end{equation*}
The same statement holds for the equation $(-i \rd_{t} \pm \brk{\nb}) \phi_{\pm} = F_{\pm}$ as well.
\end{lemma}

Finally, we end this subsection with a simple yet useful observation. Suppose that $N(\phi_{1}, \cdots, \phi_{k})$ is a $k$-linear form such that if $\phi_{j} = \phi'_{j}$ on $S_{T}$, then $N(\phi_{1}, \cdots, \phi_{k}) = N(\phi'_{1}, \cdots, \phi'_{k})$. Then for any choice of signs, $s$'s and $b$'s, it is clear, by the definition of $X^{s,b}_{\pm}(S_{T})$ and an approximation argument, that a global estimate
\begin{equation*}
	\nrm{N(\phi_{1}, \cdots, \phi_{k})}_{X^{s_{0}, b_{0}}_{\pm_0}} \leq C \prod_{j=1}^{k} \nrm{\phi_{j}}_{X^{s_{j}, b_{j}}_{\pm_j}}
\end{equation*}
for Schwartz $\phi_{1}, \cdots, \phi_{k} \in \calS(\bbR^{1+d})$ implies the local-in-time estimate
\begin{equation*}
	\nrm{N(\phi_{1}, \cdots, \phi_{k})}_{X^{s_{0}, b_{0}}_{\pm_0}(S_{T})} \leq C \prod_{j=1}^{k} \nrm{\phi_{j}}_{X^{s_{j}, b_{j}}_{\pm_j}(S_{T})}
\end{equation*}
for $\phi_{j} \in X^{s_{j}, b_{j}}_{\pm}(S_{T})$, with the same constant $C$. Thus, for the purpose of proving nonlinear estimates, it is usually possible to work just with the original $X^{s,b}_{\pm}, \,H^{s,b}$ norms and Schwartz functions.

\subsection{Bilinear null forms}
Given a bilinear form of two half-waves $\phi_{1}, \phi_{2}$ (henceforth referred to as the \emph{inputs}) on the right-hand side of a wave equation (the left-hand side being $\Box \phi$ or $(i \rd_{t} \pm \abs{\nb})\phi$), the most dangerous interaction happens when the two inputs give rise to an output which is close to the light cone in the space-time Fourier space; such an interaction is referred to as a \emph{resonant} interaction. 

As the space-time Fourier transform of a half-wave is supported on the light cone, we see that this happens \emph{if and only if} the two inputs are collinear. This means that if a bilinear form $\calB(\phi_{1}, \phi_{2})$ possesses cancellation in the angle between the two inputs, then we expect it to exhibit better regularity properties. This notion goes under the name of \emph{null structure}, due to Klainerman-Machedon \cite{Klainerman:ei}, and a bilinear form which has such structure is called a \emph{null form}. Note that if the two half-waves are of different signs, then the relevant angle is $\angle(\xi_{1}, - \xi_{2})$, not $\angle(\xi_{1}, \xi_{2})$. With these considerations in mind, we make the following definition, which is useful for treating different null forms in a unified fashion.

\begin{definition}[Abstract bilinear null form] \label{def:absNf}
Given $\xi_{1}, \xi_{2} \in \bbR^{2}$, let us denote the (smaller) angle between $\xi_{1}, \xi_{2}$ by $\angle(\xi_{1}, \xi_{2})$. For arbitrary signs $\pm_{1}, \pm_{2}$ and $\ell \in \bbR$, we define the \emph{abstract bilinear null form} of order $\ell$, denoted by $\mathfrak{B}^{\ell}_{\pm_{1}, \pm_{2}}$, by the following formula.
\begin{equation*}
	\sptF{\mathfrak{B}^{\ell}_{\pm_{1}, \pm_{2}}(\phi_{1}, \phi_{2})} (\tau_{0}, \xi_{0}) := \int_{\tau_{0} = \tau_{1} + \tau_{2}, \xi_{0} = \xi_{1} + \xi_{2}} \abs{\angle(\pm_{1} \xi_{1}, \pm_{2} \xi_{2})}^{\ell} \, \abs{\widetilde{\phi_{1}}}(\tau_{1}, \xi_{1}) \abs{\widetilde{\phi_{2}}}(\tau_{2}, \xi_{2}) \, \ud \tau_{1} \ud \xi_{1}.
\end{equation*}
The inputs $\phi_{1}, \phi_{2}$ may be any $\bbC^{n}$-valued functions.
\end{definition}

The order $\ell$ refers to the power of the angle between the two inputs in the symbol of $\mathfrak{B}^{\ell}_{\pm_{1}, \pm_{2}}$. Thus the higher the order, the `better' the cancellation is, in the sense that
\begin{equation*}
	\calB^{\ell}_{\pm_{1}, \pm_{2}}(\phi_{1}, \phi_{2}) \Fleq \calB^{k}_{\pm_{1}, \pm_{2}}(\phi_{1}, \phi_{2})
\end{equation*}
provided that $k \leq \ell$.

Let us consider bilinear forms $\nf^{\mu \nu}_{\pm_{1}, \pm_{2}}, \nf^{0}_{\pm_{1}, \pm_{2}}$ defined by
\begin{align*}
	& \nf^{\mu \nu}_{\pm_{1}, \pm_{2}} (\phi_{1, \pm_{1}}, \phi_{2, \pm_{2}}) 
	:= \mR^{\mu}_{\pm_{1}} \phi_{1, \pm_{1}} \mR^{\nu}_{\pm_{2}} \phi_{2, \pm_{2}} - \mR^{\nu}_{\pm_{1}} \phi_{1, \pm_{1}} \mR^{\mu}_{\pm_{2}} \phi_{2, \pm_{2}}, \\
	& \nf^{0}_{\pm_{1}, \pm_{2}} (\phi_{1, \pm_{1}}, \phi_{2, \pm_{2}}) 
	:= \mR_{\pm_{1},\mu} \phi_{1, \pm_{1}} \mR^{\mu}_{\pm_{2}} \phi_{2, \pm_{2}}.
\end{align*}
The notation is chosen according to the similarity of these two bilinear forms to the standard null forms $Q_{\mu \nu}, Q_{0}$ (see \cite{Klainerman:ei}), which are defined by replacing the modified Riesz transforms $\mR^{\mu}_{\pm}$ by $\rd^{\mu}$.
The following lemma states that $\nf^{\mu \nu}_{\pm_{1}, \pm_{2}}$, $\nf^{0}_{\pm_{1}, \pm_{2}}$, as well as a certain bilinear form of two spinors, have null structure in the sense that they are dominated by a abstract null form $\calB^{\ell}_{\pm_{1}, \pm_{2}}$ in the space-time Fourier space. These encompass, in fact, all null forms which arise in this work.

\begin{lemma} \label{lem:nullform}
Let $\phi_{1}, \phi_{2}$ be complex-valued Schwartz functions, and $\psi_{1}, \psi_{2}$ Schwartz spinor fields (i.e. $\bbC^{2}$-valued functions). Then the following statements hold.
\begin{enumerate}
\item $\displaystyle{\nf^{\mu \nu}_{\pm_{1}, \pm_{2}}(\phi_{1}, \phi_{2}) \Fleq \calB^{1}_{\pm_{1}, \pm_{2}}(\phi_{1}, \phi_{2})}$. \\
\item $\displaystyle{\nf^{0}_{\pm_{1}, \pm_{2}}(\phi_{1}, \phi_{2}) \Fleq \calB^{2}_{\pm_{1}, \pm_{2}}(\phi_{1}, \phi_{2})}$. \\
\item $\displaystyle{(\Pi_{\pm_{1}} \psi_{1})^{\dagger} (\Pi_{\mp_{2}} \alp^{\mu} \Pi_{\pm_{2}} \psi_{2}) \Fleq \calB^{1}_{\pm_{1}, \pm_{2}}(\psi_{1}, \psi_{2})}$. \\
\end{enumerate}
\end{lemma}
\begin{proof} 
Let us begin with (1). In the case $\mu = 0$, $\nu=i$ (where $i$ runs over $1,2$), we compute
 \begin{align*}
	&\sptF{\nf^{0i}_{\pm_{1},\pm_{2}}(\phi_{1,\pm_{1}}, \phi_{2,\pm_{2}})}(\tau_{0}, \xi_{0}) \\
	&\quad = \int_{\tau_{0} = \tau_{1} + \tau_{2}, \xi_{0} = \xi_{1} + \xi_{2}} \underbrace{\bb( \frac{\pm_{1} \xi_{1}^{i}}{\abs{\xi_{1}}}  - \frac{\pm_{2} \xi_{2}^{i}}{\abs{\xi_{2}}} \bb)}_{O(\angle(\pm_{1} \xi_{1}, \pm_{2} \xi_{2}))} 
	\widetilde{\phi_{1,\pm_{1}}}(\tau_{1}, \xi_{1}) \widetilde{\phi_{2,\pm_{2}}}(\tau_{2}, \xi_{2}) \, \ud \tau_{1} \ud \xi_{1}, 
\end{align*}
whereas in the case $\mu = i$, $\nu=j$ ($i, j$ run over $1,2$), 
\begin{align*}
	&\sptF{\nf^{ij}_{\pm_{1},\pm_{2}}(\phi_{1,\pm_{1}}, \phi_{2,\pm_{2}})}(\tau_{0}, \xi_{0}) \\
	&\quad = \int_{\tau_{0} = \tau_{1} + \tau_{2}, \xi_{0} = \xi_{1} + \xi_{2}} \underbrace{\bb( \frac{(\pm_{1} \xi_{1}^{i})(\pm_{2} \xi_{2}^{j})  - (\pm_{1} \xi_{1}^{j})(\pm_{2} \xi_{2}^{i})}{\abs{\xi_{1}} \abs{\xi_{2}}} \bb)}_{O(\angle(\pm_{1} \xi_{1}, \pm_{2} \xi_{2}))} 
	\widetilde{\phi_{1,\pm_{1}}}(\tau_{1}, \xi_{1}) \widetilde{\phi_{2,\pm_{2}}}(\tau_{2}, \xi_{2}) \, \ud \tau_{1} \ud \xi_{1}.
\end{align*}
This proves (1). For (2), we compute
\begin{align*}
	&\sptF{\nf^{0}_{\pm_{1},\pm_{2}}(\phi_{1,\pm_{1}}, \phi_{2,\pm_{2}})}(\tau_{0}, \xi_{0}) \\
	&\quad = \int_{\tau_{0} = \tau_{1} + \tau_{2}, \xi_{0} = \xi_{1} + \xi_{2}} \underbrace{\bb( 1 - \frac{(\pm_{1} \xi_{1}) \cdot (\pm_{2} \xi_{2})}{\abs{\xi_{1}} \abs{\xi_{2}}} \bb)}_{O(\angle(\pm_{1} \xi_{1}, \pm_{2} \xi_{2})^{2})} \widetilde{\phi_{1,\pm_{1}}}(\tau_{1}, \xi_{1}) \widetilde{\phi_{2,\pm_{2}}}(\tau_{2}, \xi_{2}) \, \ud \tau_{1} \ud \xi_{1}. 
\end{align*}
In order to prove (3), we need the following computation due to D'Ancona-Foschi-Selberg. 
\begin{lemma}[{\cite[Lemma 2]{DAncona:2007ex}}] \label{lem:angle4Dirac}
Let $z \in \bbC^{2}$ and $\xi_{1}, \xi_{2} \in \bbR^{2}$. Then the following inequality holds.
\begin{equation*}
	\abs{\Pi(\xi_{1}) \Pi(-\xi_{2}) z} \leq C \abs{z} \angle(\xi_{1}, \xi_{2}).
\end{equation*}
\end{lemma}

With this lemma in hand, we compute
\begin{align*}
& \sptF{(\Pi_{\pm_{1}} \psi_{1}^{\dagger} \Pi_{\mp_{2}} \alp^{\mu} \Pi_{\pm_{2}} \psi_{2})}(\tau_{0}, \xi_{0})  \\
&\quad =  \int_{\tau_{0} = \tau_{1} + \tau_{2}, \xi_{0} = \xi_{1} + \xi_{2}} (\Pi(\pm_{1} \xi_{1}) \widetilde{\psi_{1}}(\tau_{1}, \xi_{1}))^{\dagger} (\Pi(\mp_{2} \xi_{2}) \alp^{\mu} \Pi(\pm_{2} \xi_{2}) \widetilde{\psi_{2}}(\tau_{2}, \xi_{2})) \, \ud \tau_{1} \ud \xi_{1}  \\
&\quad =  \int_{\tau_{0} = \tau_{1} + \tau_{2}, \xi_{0} = \xi_{1} + \xi_{2}} (\widetilde{\psi_{1}}(\tau_{1}, \xi_{1}))^{\dagger} (\underbrace{\Pi(\pm_{1} \xi_{1}) \Pi(\mp_{2} \xi_{2}) \alp^{\mu} \, \Pi(\pm_{2} \xi_{2})}_{O(\angle(\pm_{1} \xi_{1}, \pm_{2} \xi_{2}))} \widetilde{\psi_{2}}(\tau_{2}, \xi_{2})) \, \ud \tau_{1} \ud \xi_{1}.
\end{align*}
where on the last line, we used Lemma \ref{lem:angle4Dirac}. \qedhere
\end{proof}

We remark that the null form (3) of Lemma \ref{lem:nullform} has been first introduced by P. D'Ancona, D. Foschi and S. Selberg \cite{DAncona:2007ex} in the context of the Dirac-Klein-Gordon equations. 

\section{Chern-Simon-Dirac equations: Proof of Theorem \ref{thm:lwp4CSD}}
In this section, we will prove our main theorem for the Chern-Simons-Dirac equations \eqref{eq:CSD:original}, which is Theorem \ref{thm:lwp4CSD}. 
For simplicity, we set $\kpp = 1$ in \eqref{eq:CSD:original}. Writing out $F_{\mu \nu}$, $\covD_{\mu}$ in terms of $A_{\mu}$ and using the $\bt,\, \alp^{\mu}$ matrices, we can reformulate the \eqref{eq:CSD:original} system as follows.
\begin{equation} \label{eq:CSD}
\left\{
\begin{aligned}
	& \rd_{\mu} A_{\nu} - \rd_{\nu} A_{\mu} = - 2 \eps_{\mu \nu \lmb} \psi^{\dagger} \alp^{\lmb} \psi, \\
	& i \rd_{0} \psi + i \alp^{j} \rd_{j} \psi = m \bt \psi - \alp^{\mu} A_{\mu} \psi.
\end{aligned}
\right.
\end{equation}
Let us furthermore impose the Lorenz gauge condition 
\begin{equation} \label{eq:Lorenz}
	\rd^{\mu} A_{\mu} = 0.
\end{equation}
Taking $\rd^{\mu}$ of the first equation in \eqref{eq:CSD}, we arrive at the following system of coupled wave and Dirac equations which describes the time evolution of the variables $(A_{\mu}, \psi)$.
\begin{equation} \label{eq:CSD:wave}
\left\{
\begin{aligned}
	& \Box A_{\nu} = \rd^{\mu} \mathfrak{N}_{\mu \nu}(\psi, \psi), \\
	& i \rd_{0} \psi + i \alp^{j} \rd_{j} \psi = m \bt \psi + \mathfrak{M}(\psi, A),
\end{aligned}
\right.
\end{equation}
where
\begin{equation*}
	\mathfrak{N}_{\mu \nu}(\psi_{1}, \psi_{2}) := - 2 \eps_{\mu \nu \lmb} (\psi_{1}^{\dagger} \alp^{\lmb} \psi_{2}), \qquad 
	\mathfrak{M}(\psi, A) := - \alp^{\mu} A_{\mu} \psi.
\end{equation*}

The initial data for the system \eqref{eq:CSD:wave} at $t=0$ is given in terms of $(a_{0}, a_{1}, a_{2}, \psi_{0})$ as follows.
\begin{equation} \label{eq:CSD:wave:id}
\left\{
\begin{aligned}
	&A_{\mu}(0, x) = a_{\mu}(x), \quad \psi(0, x) = \psi_{0}(x), \\
	&\rd_{t} A_{0}(0,x) = - \rd^{\ell} a_{\ell}(x), \quad \rd_{t} A_{j}(0, x) = \rd_{j} a_{0}(x) - 2 \eps_{0 j k} \psi_{0}^{\dagger} \alp^{k} \psi_{0}.
\end{aligned}
\right.
\end{equation}
Note that \eqref{eq:CSD:wave:id} is satisfied for an initial data set for the original \eqref{eq:CSD:original} system in the Lorenz gauge. Furthermore, in the converse direction, it can be easily verified that a sufficiently smooth\footnote{We remark that this statement can be extended to $C_{t} H^{s}_{x}$ solutions considered in Theorem \ref{thm:lwp4CSD} by continuous dependence on the initial data, which follows from the Picard iteration scheme we set up below.} solution $(A_{\nu}, \psi)$ to \eqref{eq:CSD:wave}, where $(a_{\mu}, \psi_{0})$ satisfies the constraint equation \eqref{eq:CSD:constraint} and the initial data are given by \eqref{eq:CSD:wave:id}, is a solution to the \eqref{eq:CSD:original} system in the Lorenz gauge. Therefore, in order to prove local well-posedness of \eqref{eq:CSD:original} under the Lorenz gauge condition, it suffices to study the system \eqref{eq:CSD:wave}.

\subsection{Set-up} \label{subsec:CSD:setup}
In this subsection, we will set up a Picard iteration in the space of type $X^{s,b}_{\pm}(S_{T})$ to prove Theorem \ref{thm:lwp4CSD}. 

We shall begin by reformulating \eqref{eq:CSD:wave} in an integral form. First, projecting the equation for $\psi$ using $\Pi_{\pm}$, we obtain
\begin{equation*} 
- (-i \rd_{0} \pm \abs{\nb}) \psi_{\pm} = m \bt \psi_{\mp} + \Pi_{\pm} \mathfrak{M}(\psi, A).
\end{equation*}
Then it is clear, by applying Lemma \ref{lem:mDuhamel} to the equation for $A_{\nu}$ and Duhamel's principle to the preceding equation, that \eqref{eq:CSD:wave} is equivalent to the following system of integral equations (for both choices of sign):
\begin{equation} \label{eq:CSD:wave:Duhamel}
\left\{
\begin{aligned}
	& A_{\nu, \pm}(t, x) = \Ahom_{\nu, \pm}(t,x) 
				- \frac{1}{2} \int_{0}^{t} e^{\pm (-i) (t-s) \abs{\nb}} \mR^{\mu}_{\pm} \mathfrak{N}_{\mu \nu}(\psi, \psi)  (s,x)\, \ud s, \\
	& \psi_{\pm}(t,x) = \psihom_{\pm}(t,x) 
				- i \int_{0}^{t} e^{\pm (-i) (t-s) \abs{\nb}} (m \bt \psi_{\mp} + \Pi_{\pm} \mathfrak{M}(\psi, A)) (s,x) \, \ud s,
\end{aligned}
\right.
\end{equation}
where $A_{\nu} = \sum_{\pm} A_{\nu, \pm}$, $\psi = \sum_{\pm} \psi_{\pm}$. The homogeneous parts $\Ahom_{\nu, \pm},\, \psihom_{\pm}$ are given by 
\begin{align*}
	& \Ahom_{\nu, \pm}  (t,x) = \frac{1}{2} e^{\pm(-i) t \abs{\nb}} \bb( A_{\nu}(0, x) \pm \frac{1}{i \abs{\nb}} (- \eps_{0 \nu \lmb} \psi^{\dagger} \alp^{\lmb} \psi - \rd_{0} A_{\nu} )(0, x) \bb), \\
	& \psihom_{\pm} (t,x) = e^{\pm (-i) t \abs{\nb}} \psi(0,x).
\end{align*}
Considering \eqref{eq:CSD:wave:id}, these are given in terms of $(a_{0}, a_{1}, a_{2}, \psi_{0})$ by
\begin{equation} \label{eq:CSD:wave:Duhamel:id}
\left\{
\begin{aligned}
	& \Ahom_{0, \pm} (t,x) = \frac{1}{2} e^{\pm(-i) t \abs{\nb}} \bb( a_{0} \pm \frac{\rd^{\ell} }{i \abs{\nb}} a_{\ell} \bb)(x), \\
	& \Ahom_{j, \pm}  (t,x) = \frac{1}{2} e^{\pm(-i) t \abs{\nb}} \bb( a_{j} \mp \frac{\rd_{j}}{i \abs{\nb}} a_{0}  \bb)(x), \\
	& \psihom_{\pm} (t,x) = e^{\pm (-i) t \abs{\nb}} \psi_{0}(x).
\end{aligned}
\right.
\end{equation}

We are now ready to set up our Picard iteration. Set $A^{(0)}_{\pm} := \Ahom_{\pm}$, $\psi^{(0)}_{\pm} := \psihom_{\pm}$, and for $n \geq 1$ define $A^{(n)}_{\pm}, \psi^{(n)}_{\pm}$ as 
\begin{equation*}
\left\{
\begin{aligned}
	& A^{(n)}_{\nu, \pm}(t, x) = \Ahom_{\nu, \pm}(t,x) 
				- \frac{1}{2} \int_{0}^{t} e^{\pm (-i) (t-s) \abs{\nb}} \mR^{\mu}_{\pm} \mathfrak{N}_{\mu \nu}(\psi^{(n-1)}, \psi^{(n-1)})  (s,x)\, \ud s, \\
	& \psi^{(n)}_{\pm}(t,x) = \psihom_{\pm}(t,x) 
				- i \int_{0}^{t} e^{\pm (-i) (t-s) \abs{\nb}} (m \bt \psi_{\mp} + \Pi_{\pm} \mathfrak{M}(\psi^{(n-1)}, A^{(n-1)})) (s,x) \, \ud s,
\end{aligned}
\right.
\end{equation*}
where $A_{\nu}^{(n)} = \sum_{\pm} A^{(n)}_{\nu, \pm}$, $\psi^{(n)} = \sum_{\pm} \psi^{(n)}_{\pm}$.
The aim is to show that $(A^{(n)}_{+}, A^{(n)}_{-}, \psi^{(n)}_{+}, \psi^{(n)}_{-})$ is a Cauchy sequence in the space $X^{s,b}_{+} \times X^{s,b}_{-} \times X^{s,b}_{+} \times X^{s,b}_{-}$.

As a first step, let us estimate the homogeneous parts (or the zeroth iterate) $\Ahom_{\mu, \pm}$ and $\psihom_{\pm}$. Since the Riesz transform $\rd_{j} / i \abs{\nb}$ is obvious bounded on $H^{s}_{x}$,  it follows from Lemma \ref{lem:XsbHom} and \eqref{eq:CSD:wave:Duhamel:id} that for any $s \in \bbR$ and $b > 1/2$,
\begin{equation*}
	\nrm{\Ahom_{\nu, \pm}}_{X^{s,b}_{\pm}(S_{T})} \aleq \sum_{\mu} \nrm{a_{\mu}}_{H^{s}_{x}}, \quad 
	\nrm{\psihom_{\pm}}_{X^{s,b}_{\pm}(S_{T})} \aleq \nrm{\psi_{0}}_{H^{s}_{x}} ,
\end{equation*}
where we note that the implicit constants do \emph{not} depend on $T > 0$. 

Next, in order to prove that the above Picard iteration converges, by standard arguments, it suffices to establish the following estimates, for some $\eps_{0} > 0$ and arbitrary spinor fields $\psi_{1}, \,\psi_{2},\, \psi$ and 1-form $A_{\nu}$, all of which are Schwartz:
\begin{align} 
	& \nrm{\mR_{\pm_{0}}^{\nu} \mathfrak{N}_{\mu \nu}(\psi_{1}, \psi_{2})}_{X^{s,b-1+\eps_{0}}_{\pm_{0}}} 
	\aleq \sum_{\pm_{1}, \pm_{2}} \nrm{\psi_{1, \pm_{1}}}_{X^{s,b}_{\pm_{1}}} \nrm{\psi_{2, \pm_{2}}}_{X^{s,b}_{\pm_{2}}}, \label{eq:CSD:est4N} \\
	& 	\nrm{m \bt \psi_{\mp}}_{X^{s,b-1+\eps_{0}}_{\pm_{0}}} 
	\aleq m \nrm{\psi_{\mp}}_{X^{s,b}_{\mp}}, \label{eq:CSD:est4M:0} \\
	&\nrm{\Pi_{\pm_{0}} \mathfrak{M}(\psi, A_{\nu})}_{X^{s,b-1+\eps_{0}}_{\pm_{0}}} 
	\aleq  \sum_{\pm_{1}, \pm_{2}} \nrm{\psi_{\pm_{1}}}_{X^{s,b}_{\pm_{1}}} \nrm{\bfA_{\pm_{2}}}_{X^{s,b}_{\pm_{2}}}, \label{eq:CSD:est4M:1}
\end{align}
where $\nrm{\bfA_{\pm}}_{X^{s,b}_{\pm}}$ is a shorthand for $\sum_{\nu=0,1,2} \nrm{A_{\nu, \pm}}_{X^{s,b}_{\pm}}$.

Note that the above global estimates actually immediately imply the corresponding estimates for the restriction spaces $X^{s,b}_{\pm}(S_{T})$; see the remark at the end of \S \ref{subsec:prelim:Xsb}. Moreover, thanks to the presence of $\eps_{0} > 0$, we gain a factor of $\abs{T}$ when applying Lemma \ref{lem:XsbInhom} to estimate the $X^{s,b}_{\pm}$ norm of the inhomogeneous parts, which can be used to obtain the necessary smallness to make the Picard iterates $(A^{(n)}_{+}, A^{(n)}_{-}, \psi^{(n)}_{+}, \psi^{(n)}_{-})$ form a Cauchy sequence. Finally, note that Lemma \ref{lem:Xsb2Hs} ensures that the solution that we obtain belongs to $C_{t} ((-T, T), H^{s}_{x})$, as desired.
As the arguments from this point to $H^{s}$-local well-posedness (Theorem \ref{thm:lwp4CSD}) are quite standard, we will omit the details. Henceforth, our focus will be to establish the global estimates \eqref{eq:CSD:est4N}--\eqref{eq:CSD:est4M:1}.

\subsection{Null structure of the Chern-Simons-Dirac in the Lorenz gauge} \label{subsec:CSD:nullform}
In this subsection, we will reveal the null structure of the quadratic nonlinearities of \eqref{eq:CSD:original} in the Lorenz gauge.

\subsubsection*{Wave equation for $A_{\mu}$}
Applying the commutator identity \eqref{eq:commAlpPi}, $\mathfrak{N}_{\mu \nu}$ may be written as
\begin{equation*}
\mathfrak{N}_{\mu \nu}(\psi_{1}, \psi_{2})
= \mathfrak{N}_{\mu \nu, 1}(\psi_{1}, \psi_{2}) + \mathfrak{N}_{\mu \nu, 2} (\psi_{1}, \psi_{2}), 
\end{equation*}
where
\begin{align*}
& \mathfrak{N}_{\mu \nu, 1}(\psi_{1}, \psi_{2}) := - 2 \sum_{\pm_{1}, \pm_{2}} \eps_{\mu \nu \lmb} (\psi_{1, \pm_{1}}^{\dagger} \Pi_{\mp_{2}} (\alp^{\lmb} \psi_{2, \pm_{2}})), \\
& \mathfrak{N}_{\mu \nu, 2}(\psi_{1}, \psi_{2}) := 2 \sum_{\pm_{1}, \pm_{2}} \eps_{\mu \nu \lmb} (\psi_{1, \pm_{1}}^{\dagger} \mR^{\lmb}_{\pm_{2}} \psi_{2, \pm_{2}}).
\end{align*}

For the term $\mathfrak{N}_{\mu \nu, 1}$, we will prove
\begin{equation} \label{eq:CSD:nf4N1}
	\nrm{\eps_{\mu \nu \lmb} \mR_{\pm_{0}}^{\nu} (\psi_{1, \pm_{1}}^{\dagger} \Pi_{\mp_{2}} (\alp^{\lmb} \psi_{2, \pm_{2}})) }_{X^{s,b-1+\eps_{0}}_{\pm_{0}}} 
	\aleq \nrm{\psi_{1, \pm_{1}}}_{X^{s,b}_{\pm_{1}}} \nrm{\psi_{2, \pm_{2}}}_{X^{s,b}_{\pm_{2}}}
\end{equation}
for arbitrary signs $(\pm_{0}, \,\pm_{1},\, \pm_{2})$ and spinor fields $\psi_{1, \pm_{1}},\, \psi_{2, \pm_{2}}$, both of which are Schwartz.
Note that Part (3) of Lemma \ref{lem:nullform} applies to this bilinear form. 

On the other hand, the null form of the second term $\mathfrak{N}_{\mu \nu,2}$ is most easily seen after a duality argument. For this term, we will prove
\begin{equation} \label{eq:CSD:nf4N2}	
	\bb\vert \int \eps_{\mu \nu \lmb} (\psi_{1, \pm_{1}}^{\dagger} \mR^{\lmb}_{\pm_{2}} \psi_{2, \pm_{2}}) \overline{\mR_{\pm_{0}}^{\nu} \phi_{\pm_{0}}} \, \ud t \ud x\bb\vert
	\aleq \nrm{\psi_{1, \pm_{1}}}_{X^{s,b}_{\pm_{1}}} \nrm{\psi_{2, \pm_{2}}}_{X^{s,b}_{\pm_{2}}} \nrm{\phi_{\pm_{0}}}_{X^{-s,1-b-\eps_{0}}_{\pm_{0}}}
\end{equation}
for every combination $(\pm_{0}, \,\pm_{1},\, \pm_{2})$ of signs, Schwartz spinor fields $\psi_{1, \pm_{1}},\, \psi_{2, \pm_{2}}$ and $\phi_{\pm_{0}} \in \calS(\bbR^{1+2})$. Note that $\mR_{\pm_{0}}^{\nu}$ has been moved to the factor $\phi_{\pm_{0}}$ using self-adjointness. We remark that the left-hand side of \eqref{eq:CSD:nf4N2} possesses a $\nf^{\lmb \nu}$-type null form between $\psi_{2,\pm_{2}}$ and $\phi_{\pm_{0}}$. 

\subsubsection*{Dirac equation for $\psi$}
Let us focus on the bilinear estimate \eqref{eq:CSD:est4M:1}. Using the commutator formula \eqref{eq:commAlpPi}, we have $- \Pi_{\pm}( A_{\mu} \alp^{\mu} \psi) = \mathfrak{M}_{\pm,1}(\psi, A) +\mathfrak{M}_{\pm,2}(\psi, A)$, where
\begin{align*}
&\mathfrak{M}_{\pm_{0},1}(\psi, A) := - \sum_{\pm} \Pi_{\pm_{0}} (A_{\mu} \Pi_{\mp} \alp^{\mu} \psi_{\pm}) , \\
& \mathfrak{M}_{\pm_{0},2}(\psi, A) :=  \sum_{\pm}  \Pi_{\pm_{0}} ( A_{\mu} \, \mR^{\mu}_{\pm} \psi_{\pm} ).
\end{align*}

The null structure of the first term $\mathfrak{M}_{\pm,1}$ may be exhibited via a duality argument, as in \cite{DAncona:2007ex}. For this term, we will establish
\begin{equation} \label{eq:CSD:nf4M1}
	\bb\vert \int A_{\mu, \pm_{2}} (\psi^{\dagger}_{0, \pm_{0}} \Pi_{\mp_{1}} \alp^{\mu} \psi_{1,\pm_{1}})  \, \ud t \ud x \bb\vert \aleq \nrm{\psi_{0, \pm_{0}}}_{X^{-s,1-b-\eps_{0}}_{\pm_{0}}} \nrm{\psi_{1, \pm_{1}}}_{X^{s,b}_{\pm_{1}}} \nrm{\bfA_{\pm_{2}}}_{X^{s,b}_{\pm_{2}}} 
\end{equation}
for arbitrary signs $(\pm_{0},\, \pm_{1},\, \pm_{2})$, spinor fields $\psi_{0, \pm_{0}},\, \psi_{1, \pm_{1}}$ and 1-form $A_{\mu, \pm_{2}}$, all of which are Schwartz. Note that the bilinear form involving $\psi_{0,\pm_{0}}$ and $\psi_{1, \pm_{1}}$ is a null form according to Part (3) of Lemma \ref{lem:nullform}.

%

Finally, in order to reveal the null structure of the second term $\mathfrak{M}_{\pm_{0},2}$, we proceed as in \cite{Selberg:2010ig}. The first step is to divide $A_{j} = A^{\df}_{j} + A^{\cf}_{j}$ according to the Hodge decomposition. To describe this decomposition, let us consider a 1-form $A$ which is Schwartz in space, and recall the definitions $\curl A := \rd_{1} A_{2} - \rd_{2} A_{1}$ and $\div A := - \rd_{1} A_{1} - \rd_{2} A_{2}$. Then the \emph{Hodge decomposition theorem} states that $A_{i}$ may be (uniquely) written as a sum of \emph{divergence-} and \emph{curl-free} parts $A^{\df}_{i} + A^{\cf}_{i}$ which vanish sufficiently fast at the spatial infinity. The latter two 1-forms are given by the formulae
\begin{align*}
	A^{\df}_{1} \ud x^{1} + A^{\df}_{2} \ud x^{2} 
	:= & (-\lap)^{-1} (\rd_{2} (\curl A) \, \ud x^{1} - \rd_{1} (\curl A) \, \ud x^{2}), \\
	A^{\cf}_{1} \ud x^{1} + A^{\cf}_{2} \ud x^{2} 
	:= & (-\lap)^{-1} (\rd_{1} (\div A) \, \ud x^{1} + \rd_{2} (\div A) \, \ud x^{2}).
\end{align*}
The above statements can be readily verified using the Fourier transform.
It is then a well-known fact that $A^{\df}_{\ell} \, \mR^{\ell}_{\pm} \psi_{\pm}$ possesses a null structure. Indeed, we have the component-wise formula
\begin{equation*}
	\sum_{\pm_{1}} A^{\df}_{\ell} \, \mR^{\ell}_{\pm_{1}} \psi_{\pm_{1}} 
	= \sum_{\pm_{1}, \pm_{2}} \nf^{12}_{\pm_{2}, \pm_{1}}(B_{\pm_{2}}, \psi_{\pm_{1}}),
\end{equation*}
where
\begin{equation*}
	B_{\pm} := \mR_{\pm,1} A_{2,\pm} - \mR_{\pm,2} A_{1,\pm}.
\end{equation*}


On the other hand, using the fact that $\rd^{\ell} A_{\ell} = - \rd_{0} A_{0} = i (\abs{\nb} A_{0,+} - \abs{\nb} A_{0,-})$, we have
\begin{align*}
& A_{0} = - (\mR_{+, 0} A_{0,+} + \mR_{-, 0} A_{0, -} ), \\
& A^{\cf}_{j} = - ( \mR_{+, j} A_{0,+} + \mR_{-, j}A_{0, -}).
\end{align*}
Therefore, we obtain a component-wise formula
\begin{equation*}
	\sum_{\pm_{1}} \bb( A_{0}  \, \mR^{0}_{\pm_{1}} \psi_{\pm_{1}} + A^{\cf}_{\ell} \, \mR^{\ell}_{\pm_{1}} \psi_{\pm_{1}} \bb)
	= \sum_{\pm_{1}, \pm_{2}} \nf^{0}_{\pm_{1}, \pm_{2}} (\psi_{\pm_{1}}, A_{0, \pm_{2}}).
\end{equation*}
In sum, we will establish 
\begin{align} 
	\nrm{\Pi_{\pm_{0}}(\nf^{12}_{\pm_{2}, \pm_{1}}(B_{\pm_{2}}, \psi_{\pm_{1}}))}_{X^{s, b-1+\eps_{0}}_{\pm_{0}}} 
	\aleq &  \nrm{\psi_{\pm_{1}}}_{X^{s,b}_{\pm_{1}}} \nrm{B_{\pm_{2}}}_{X^{s,b}_{\pm_{2}}} , \label{eq:CSD:nf4M2:1} \\
	\nrm{\Pi_{\pm_{0}} (\nf^{0}_{\pm_{1}, \pm_{2}} (\psi_{\pm_{1}}, A_{0, \pm_{2}}))}_{X^{s, b-1+\eps_{0}}_{\pm_{0}}} 
	\aleq &  \nrm{\psi_{\pm_{1}}}_{X^{s,b}_{\pm_{1}}} \nrm{A_{0, \pm_{2}}}_{X^{s,b}_{\pm_{2}}} \label{eq:CSD:nf4M2:2}
\end{align}
for all combinations of signs $(\pm_{0}, \pm_{1}, \pm_{2})$ and Schwartz inputs.

\subsection{Proof of the global estimates} \label{subsec:CSD:inhom}
Based on the null structure revealed in the previous subsection, we will find sufficient conditions on $(s, b)$ to rigorously establish \eqref{eq:CSD:est4N}--\eqref{eq:CSD:est4M:1}.

\subsubsection*{Linear estimate}
Let us begin by proving \eqref{eq:CSD:est4M:0}, which is the (linear) mass term for the Dirac equation. Provided that
\begin{equation} \label{eq:CSD:cond4sb:1}
	0 < b < 1,
\end{equation}
and $\eps_{0}$ is small enough, we simply compute
\begin{align*}
	\nrm{\Pi_{\pm_{0}} (m \bt \psi_{\mp})}_{X^{s,b-1+\eps_{0}}_{\pm_{0}}} 
	\leq m \nrm{\psi_{\mp}}_{X^{s, b-1+\eps_{0}}_{\pm_{0}}} 
	\leq m \nrm{\brk{\nb}^{s} \psi_{\mp}}_{L^{2}_{t,x}} 
	\leq m \nrm{\psi_{\mp}}_{X^{s,b}_{\mp}}. 
\end{align*}

\subsubsection*{Bilinear estimates}
Next, we will prove \eqref{eq:CSD:est4N} and \eqref{eq:CSD:est4M:1}, which have been reduced to proving \eqref{eq:CSD:nf4N1}, \eqref{eq:CSD:nf4N2}, \eqref{eq:CSD:nf4M1}, \eqref{eq:CSD:nf4M2:1} and \eqref{eq:CSD:nf4M2:2} for every combination of signs $(\pm_{0},\pm_{1},\pm_{2})$ and Schwartz inputs. Applying Lemma \ref{lem:nullform}, we see that all of these estimates would follow if we establish
\begin{align}
	& \nrm{\mathfrak{B}^{1}_{\pm_{1}, \pm_{2}}(\phi_{1}, \phi_{2}) }_{X^{s, b-1+\eps_{0}}_{\pm_{0}}}
	\aleq \nrm{\phi_{1}}_{X^{s,b}_{\pm_{1}}} \nrm{\phi_{2}}_{X^{s,b}_{\pm_{2}}},  \label{eq:nfEst:1:0} \\
	& \nrm{\mathfrak{B}^{1}_{\pm_{1}, \pm_{2}}(\phi_{1}, \phi_{2}) }_{X^{-s, -b}_{\pm_{0}}}
	\aleq \nrm{\phi_{1}}_{X^{s,b}_{\pm_{1}}} \nrm{\phi_{2}}_{X^{-s,1-b-\eps_{0}}_{\pm_{2}}} \label{eq:nfEst:2:0}
\end{align}
for every combination of signs $(\pm_{0},\pm_{1},\pm_{2})$ and Schwartz functions $\phi_{1},\, \phi_{2}$. 
Note that $X^{s,b}_{\pm} \subset H^{s,b} \subset H^{s,-b} \subset X^{s,-b}_{\pm}$ for $s \in \bbR$ and $b \geq 0$. Therefore, \eqref{eq:nfEst:1:0} and \eqref{eq:nfEst:1:1} would follow from the analogous estimates in which $X^{s,b}_{\pm}$ are replaced by $H^{s,b}$, provided that \eqref{eq:CSD:cond4sb:1} holds and $\eps_{0}$ is sufficiently small.

In order to quantify the effect of the angular cancellation present in the abstract null form $\mathfrak{B}^{1}_{\pm_{1}, \pm_{2}}$, we will borrow the following lemma from \cite{Selberg:2012vb}.
\begin{lemma} \label{lem:est4angle}
Consider arbitrary signs $(\pm_{1}, \pm_{2})$, $\tau_{1}, \tau_{2} \in \bbR$, and $\xi_{1}, \xi_{2} \in \bbR^{2}$. For $\tau_{0} := \tau_{1} + \tau_{2}$ and $\xi_{0} := \xi_{1} + \xi_{2}$, the following inequality holds.
\begin{equation} \label{}
	\angle (\pm_{1} \xi_{1}, \pm_{2} \xi_{2}) 
	\aleq \bb( \frac{\brk{\abs{\tau_{0}} - \abs{\xi_{0}}} + \brk{\tau_{1} \pm_{1} \abs{\xi_{1}}} + \brk{\tau_{2} \pm_{2} \abs{\xi_{2}}}}
				{\min\set{\brk{\xi_{1}}, \brk{\xi_{2}}}} \bb)^{1/2}.
\end{equation}
\end{lemma}
\begin{proof} 
This is equivalent to \cite[Lemma 3.2]{Selberg:2012vb}.
\end{proof}

In view of Lemma \ref{lem:est4angle}, to prove \eqref{eq:nfEst:1:0}, it suffices to establish
\begin{equation} \label{eq:nfEst:1:1}
\left\{
\begin{aligned}
	& \nrm{\phi_{1} \phi_{2}}_{H^{s, b-1/2+\eps_{0}}} \aleq \nrm{\phi_{1}}_{H^{s+1/2, b}} \nrm{\phi_{2}}_{H^{s, b}}, \\
	& \nrm{\phi_{1} \phi_{2}}_{H^{s, b-1+\eps_{0}}} \aleq \nrm{\phi_{1}}_{H^{s+1/2, b-1/2}} \nrm{\phi_{2}}_{H^{s, b}}, \\
	& \nrm{\phi_{1} \phi_{2}}_{H^{s, b-1+\eps_{0}}} \aleq \nrm{\phi_{1}}_{H^{s+1/2, b}} \nrm{\phi_{2}}_{H^{s, b-1/2}},
\end{aligned}
\right.
\end{equation}
whereas for \eqref{eq:nfEst:2:0}, it is enough to prove
\begin{equation} \label{eq:nfEst:2:1}
\left\{
\begin{aligned}
	& \nrm{\phi_{1} \phi_{2}}_{H^{-s, -b+1/2}} \aleq \nrm{\phi_{1}}_{H^{s+1/2, b}} \nrm{\phi_{2}}_{H^{-s, 1-b-\eps_{0}}}, \\
	& \nrm{\phi_{1} \phi_{2}}_{H^{-s, -b+1/2}} \aleq \nrm{\phi_{1}}_{H^{s, b}} \nrm{\phi_{2}}_{H^{-s+1/2, 1-b-\eps_{0}}}, \\
	& \nrm{\phi_{1} \phi_{2}}_{H^{-s, -b}} \aleq \nrm{\phi_{1}}_{H^{s+1/2, b-1/2}} \nrm{\phi_{2}}_{H^{-s, 1-b-\eps_{0}}}, \\
	& \nrm{\phi_{1} \phi_{2}}_{H^{-s, -b}} \aleq \nrm{\phi_{1}}_{H^{s, b-1/2}} \nrm{\phi_{2}}_{H^{-s+1/2, 1-b-\eps_{0}}}, \\
	& \nrm{\phi_{1} \phi_{2}}_{H^{-s, -b}} \aleq \nrm{\phi_{1}}_{H^{s+1/2, b}} \nrm{\phi_{2}}_{H^{-s, 1/2-b-\eps_{0}}}, \\
	& \nrm{\phi_{1} \phi_{2}}_{H^{-s, -b}} \aleq \nrm{\phi_{1}}_{H^{s, b}} \nrm{\phi_{2}}_{H^{-s+1/2, 1/2-b-\eps_{0}}}.
\end{aligned}
\right.
\end{equation}
By duality, \eqref{eq:nfEst:1:1} is subsumed to \eqref{eq:nfEst:2:1}; therefore, we are reduced to proving \eqref{eq:nfEst:2:1}.

Quite conveniently, the necessary and sufficient conditions\footnote{There are still some endpoint cases missing; see \cite{DAncona:2010vz} for more detail.} for such product estimates to hold are already available in the work of d'Ancona, Foschi and Selberg \cite{DAncona:2010vz}. We summarize the result that we need in the following theorem.

\begin{theorem}[d'Ancona, Foschi and Selberg \cite{DAncona:2010vz}] \label{thm:HsbProd}
Let $s_{0}, s_{1}, s_{2}, b_{0}, b_{1}, b_{2} \in \bbR$, and $\phi_{1}, \phi_{2}$ Schwartz functions on $\bbR^{1+2}$. The product estimate
\begin{equation*}
	\nrm{\phi_{1} \phi_{2}}_{H^{-s_{0}, -b_{0}}} \leq C \nrm{\phi_{1}}_{H^{s_{1}, b_{1}}} \nrm{\phi_{2}}_{H^{s_{2}, b_{2}}}
\end{equation*}
holds for some $C > 0$ if the following conditions are satisfied.
\begin{align*}
& b_{0} + b_{1} + b_{2} > 1/2, \\
& b_{0} + b_{1} + b_{2} \geq \max \set{b_{0}, b_{1}, b_{2}}, \\
& s_{0} + s_{1} + s_{2} > 3/2 - (b_{0} + b_{1} + b_{2}), \\
& s_{0} + s_{1} + s_{2} > 1 - \min \set{b_{0}+b_{1}, b_{1}+b_{2}, b_{2}+b_{0}}, \\
& s_{0} + s_{1} + s_{2} > 1/2 - \min \set{b_{0}, b_{1}, b_{2}}, \\
& s_{0} + s_{1} + s_{2} > 3/4, \\
& (s_{0} + b_{0}) + 2s_{1} + 2s_{2} > 1, \\
& 2s_{0} + (s_{1} + b_{1}) + 2s_{2} > 1, \\
& 2s_{0} + 2s_{1} + (s_{2} + b_{2}) > 1, \\
& s_{0} + s_{1} + s_{2} \geq \max \set{s_{0}, s_{1}, s_{2}}, \\
& b_{0} + s_{1} + s_{2} \geq 0, \\
& s_{0} + b_{1} + s_{2} \geq 0, \\
& s_{0} + s_{1} + b_{2} \geq 0.
\end{align*}
\end{theorem}

Using Theorem \ref{thm:HsbProd}, it is not difficult to check that \eqref{eq:nfEst:2:1} holds provided that
\begin{equation} \label{eq:CSD:cond4sb:2}
	\frac{1}{2} < b < 1, \qquad s > \max \bb\{\frac{1}{4},\, \frac{b}{2} - \frac{1}{4},\, 1-b,\, \frac{b}{3} \bb\}
\end{equation}
and $\eps_{0} > 0$ is small enough. 

\subsubsection*{Conclusion}
Choosing $s = 1/4+\eps$ and $b = 3/4-2\eps$ for an arbitrary $\eps$ satisfying $0 < \eps \ll 1$, and choosing $\eps_{0}$ small enough, one may verify that all of the conditions \eqref{eq:CSD:cond4sb:1}, \eqref{eq:CSD:cond4sb:2} are satisfied. This proves Theorem \ref{thm:lwp4CSD}.

\section{Chern-Simons-Higgs equations: Proof of Theorem \ref{thm:lwp4CSH}}
In this final section, we apply our techniques to the Chern-Simons-Higgs equations \eqref{eq:CSH:original} in the Lorenz gauge, which leads to an improvement over the low regularity local well-posedness result in \cite{Selberg:2012vb}.
Before we begin, let us briefly indicate the new idea that is being added here. In \cite{Selberg:2012vb}, the bottleneck for the restriction $s > 3/8$ ($s$ denotes the regularity of $A_{\mu}$, whereas $\phi$ has the regularity $s+1/2$)  was the bilinear estimate\footnote{More precisely, the {\it high} $\times$ {\it high} $\to$ {\it low} interaction of $\eps_{\mu \nu \lmb}\Box^{-1} \Im(\rd^{\nu} \overline{\phi} \rd^{\lmb} \phi)$.} for $\eps_{\mu \nu \lmb}\Box^{-1} \Im(\rd^{\nu} \overline{\phi} \rd^{\lmb} \phi)$. The idea is to simply use Lemma \ref{lem:mDuhamel}, instead of using the product rule for $\rd^{\nu}$, to deal with the equation $\Box A_{\mu} = \eps_{\mu \nu \lmb} \rd^{\nu} 2 \Im(\overline{\phi} \covD^{\lmb} \phi)$. Then the unfavorable operator $\abs{\nb}^{-1}$ in the wave kernel is canceled (which had been the source of trouble), leading to the improvement $s > 1/4$.

\subsection{Set-up}
Consider the system \eqref{eq:CSH:original} under the Lorenz gauge condition $\rd^{\mu} A_{\mu} = 0$, along with an $H^{s}$ initial data set $(a_{\mu}, f, g)$. We will take $\kpp =1$ for simplicity. In order to prove Theorem \ref{thm:lwp4CSH}, we will set up a Picard iteration, using the system
\begin{equation} \label{eq:CSH:wave}
	\left\{
	\begin{aligned}
	& \Box A_{\nu} = 2 \eps_{\mu \nu \lmb} \rd^{\mu} \Im (\overline{\phi} \covD^{\lmb} \phi), \\
	& (\Box + 1) \phi= 2 i A^{\mu} \rd_{\mu} \phi + A^{\mu} A_{\mu} \phi + \phi,
	\end{aligned}
	\right.
\end{equation}
with the initial data
\begin{equation} \label{eq:CSH:wave:id}
\left\{
\begin{aligned}
	&A_{\mu}(0, x) = a_{\mu}(x), 
	\quad  \rd_{t} A_{0}(0,x) = - \rd^{\ell} a_{\ell}(x), \\
	& \rd_{t} A_{j}(0, x) = \rd_{j} a_{0}(x) + 2 \eps_{0 j k} \Im (\overline{f} (\rd^{k} - i a^{k}) f)(x), \\
	& \phi(0, x) = f(x), \quad
	 \rd_{t} \phi (0, x) = g(x).
\end{aligned}
\right.
\end{equation}
Indeed, \eqref{eq:CSH:wave} is easily derived from \eqref{eq:CSH:original} by taking $\rd^{\mu} F_{\mu \nu}$ and utilizing the Lorenz gauge condition, as in the case of \eqref{eq:CSD:original}. Moreover, a (sufficiently smooth) solution to \eqref{eq:CSH:wave}, where $(a_{\mu}, f,g)$ satisfies the constraint equation \eqref{eq:CSH:constraint} and the initial data are given by \eqref{eq:CSH:wave:id}, also solves the original \eqref{eq:CSH:original} system in the Lorenz gauge.
Thus, to prove local well-posedness of the \eqref{eq:CSH:original} under the Lorenz gauge condition, it suffices to study \eqref{eq:CSH:wave}.

Note that, as in \cite{Selberg:2012vb}, we have added $\phi$ to both sides of the wave equation for $\phi$, thereby introducing an artificial \emph{mass}. This is to avoid the operator $\abs{\nb}^{-1}$ in the half-wave decomposition (see \eqref{eq:hwDecomp4KG}, \eqref{eq:KG:hw}), which is cumbersome for low frequency. 

As before, we will apply Lemma \ref{lem:mDuhamel} to the equation for $A_{\nu}$. Moreover, applying \eqref{eq:KG:hw} and Duhamel's formula for the equation for $\phi$, we arrive at the following integral formulation of \eqref{eq:CSH:wave}.
\begin{equation*}
\begin{aligned}
	& A_{\nu, \pm}(t, x) = \Ahom_{\nu, \pm}(t,x) 
				- \int_{0}^{t} e^{\pm (-i) (t-s) \abs{\nb}} \mR^{\mu}_{\pm} (\eps_{\mu \nu \lmb} \Im (\overline{\phi} \covD^{\lmb} \phi))  (s,x)\, \ud s, \\
	& \phi_{\pm}(t,x) = \phihom_{\pm} \pm i \int_{0}^{t} \frac{e^{\pm (-i) (t-s) \brk{\nb}}}{2 i \brk{\nb}} (2 i A^{\mu} \rd_{\mu} \phi + A^{\mu} A_{\mu} \phi + \phi) (s, x) \, \ud s,
\end{aligned}
\end{equation*}
where $A_{\nu} = \sum_{\pm} A_{\nu, \pm}$, $\phi = \sum_{\pm} \phi_{\pm}$. By \eqref{eq:CSH:wave:id}, the homogeneous parts are given in terms of $(a_{\mu}, f, g)$ by
\begin{equation*}
\left\{
\begin{aligned}
	& \Ahom_{0, \pm} (t,x) = \frac{1}{2} e^{\pm(-i) t \abs{\nb}} \bb( a_{0} \pm \frac{\rd^{\ell} }{i \abs{\nb}} a_{\ell} \bb)(x), \\
	& \Ahom_{j, \pm}  (t,x) = \frac{1}{2} e^{\pm(-i) t \abs{\nb}} \bb( a_{j} \mp \frac{\rd_{j}}{i \abs{\nb}} a_{0} \bb)(x), \\
	& \phihom_{\pm} (t,x) = \frac{1}{2} e^{\pm (-i) t \brk{\nb}} \bb( f \pm \frac{1}{i \brk{\nb}} g\bb)(x).
\end{aligned}
\right.
\end{equation*}

With these equations, we may now set up a Picard iteration scheme in the space
\begin{equation*}
(A_{\nu, +}, A_{\nu, -}, \phi_{+}, \phi_{-}) \in X^{s,b}_{+} \times X^{s,b}_{-} \times X^{s+1/2,b}_{+} \times X^{s+1/2,b}_{-}
\end{equation*}
as in \S \ref{subsec:CSD:setup}. Using Lemma \ref{lem:XsbHom}, the homogeneous parts can be treated as before. By the same reduction as sketched in \S \ref{subsec:CSD:setup}, the convergence of the Picard iteration scheme (and thus Theorem \ref{thm:lwp4CSH} itself) would be a consequence of the following global estimates on $\bbR^{1+2}$ for Schwartz inputs.
\begin{align} 
\nrm{\eps_{\mu \nu \lmb} \mR^{\nu}_{\pm} (\Im (\overline{\phi_{1}} \rd^{\lmb} \phi_{2}))}_{X^{s,b-1+\eps_{0}}_{\pm}}
	\aleq & \sum_{\pm_{1}, \pm_{2}} \nrm{\phi_{1, \pm_{1}}}_{X^{s+1/2,b}_{\pm_{1}}} \nrm{\phi_{2, \pm_{2}}}_{X^{s+1/2,b}_{\pm_{2}}} \label{eq:CSH:est4N:1}, \\
\nrm{\eps_{\mu \nu \lmb} \mR^{\nu}_{\pm} (\Im (\overline{\phi_{1}} A^{\lmb} \phi_{2}))}_{X^{s,b-1+\eps_{0}}_{\pm}}
	\aleq & \sum_{\pm_{1}, \pm_{2}, \pm_{3}} \nrm{\phi_{1, \pm_{1}}}_{X^{s+1/2,b}_{\pm_{1}}} \nrm{\phi_{2, \pm_{2}}}_{X^{s+1/2,b}_{\pm_{2}}} 
						\nrm{\bfA_{\pm_{3}}}_{X^{s,b}_{\pm_{3}}} \label{eq:CSH:est4N:2}, \\
\nrm{\phi}_{X^{s-1/2,b-1+\eps_{0}}_{\pm}} 
	\aleq & \sum_{\pm_{1}} \nrm{\phi_{\pm_{1}}}_{X^{s+1/2, b}_{\pm}} \label{eq:CSH:est4M:0}, \\
\nrm{A^{\mu} \rd_{\mu} \phi}_{X^{s-1/2,b-1+\eps_{0}}_{\pm}} 
	\aleq & \sum_{\pm_{1}, \pm_{2}} \nrm{\bfA_{\pm_{1}}}_{X^{s, b}_{\pm_{1}}} \nrm{\phi_{\pm_{2}}}_{X^{s+1/2,b}_{\pm_{2}}}  \label{eq:CSH:est4M:1} , \\
\nrm{A^{\mu} A'_{\mu} \phi}_{X^{s-1/2, b-1+\eps_{0}}_{\pm}}
	\aleq & \sum_{\pm_{1}, \pm_{2}, \pm_{3}} \nrm{\bfA_{\pm_{1}}}_{X^{s,b}_{\pm_{1}}} 
						\nrm{\bfA'_{\pm_{2}}}_{X^{s,b}_{\pm_{2}}} \nrm{\phi_{\pm_{3}}}_{X^{s+1/2,b}_{\pm_{3}}} , \label{eq:CSH:est4M:2}
\end{align}
where we remind the reader that $\nrm{\bfA_{\pm}}_{X^{s,b}_{\pm}}$ is a shorthand for $\sum_{\nu=0,1,2} \nrm{A_{\nu, \pm}}_{X^{s,b}_{\pm}}$.

\subsection{Null structure of the Chern-Simons-Higgs equations in the Lorenz gauge} \label{subsec:CSH:nullform}
In this subsection, we will reveal the null structure of the quadratic nonlinearities of \eqref{eq:CSH:wave}, essentially as in \cite{Selberg:2012vb}, but cast in the language of the present paper.

In order to imitate the arguments in \S \ref{subsec:CSD:nullform}, we will use the formula
\begin{equation} \label{eq:CSH:rd2mR}
	\rd_{\mu} \phi = \sum_{\pm} (\pm i) \mR_{\pm, \mu} \abs{\nb} \phi_{\pm} - \sum_{\pm} (\pm i) \err_{\mu} \phi_{\pm},
\end{equation}
where the \emph{error operator} $\err_{\mu}$ is defined by $\err_{0} := \brk{\nb} - \abs{\nb}$, $\err_{j} = 0$. The error term $(\pm i) \err_{\mu} \phi_{\pm}$ arises due to our introduction of the (artificial) mass to the wave equation for $\phi$; however, observe that for each $\mu = 0,1,2$ and $\phi \Fgeq 0$, we have
\begin{equation} \label{eq:CSH:est4err}
0 \Fleq \err_{\mu} \phi \Fleq \phi.
\end{equation}

This indicates that $\err_{\mu}$ can essentially be ignored; i.e. the extra error terms should be of lower order and thus benign.

\subsubsection*{Wave equation for $A_{\nu}$}
Consider the estimate \eqref{eq:CSH:est4N:1} for the quadratic nonlinearity of the wave equation for $A_{\nu}$. Decomposing $\phi_{1}$ into half-waves and applying the formula \eqref{eq:CSH:rd2mR} to $\phi_{2}$, this estimate is reduced to the following estimates, for all combination of signs $(\pm_{0}, \pm_{1}, \pm_{2})$:
\begin{align}
	\nrm{\eps_{\mu \nu \lmb} \mR^{\nu}_{\pm_{0}} (\Im (\overline{\phi_{1,\pm_{1}}} \mR^{\lmb}_{\pm_{2}} \abs{\nb} \phi_{2, \pm_{2}}))}_{X^{s,b-1+\eps_{0}}_{\pm_{0}}}
	\aleq & \nrm{\phi_{1, \pm_{1}}}_{X^{s+1/2,b}_{\pm_{1}}} \nrm{\phi_{2, \pm_{2}}}_{X^{s+1/2,b}_{\pm_{2}}} \label{eq:CSH:est4N:1:1}, \\
	\nrm{\eps_{\mu \nu \lmb} \mR^{\nu}_{\pm_{0}} (\Im (\overline{\phi_{1, \pm_{1}}} \err^{\lmb} \phi_{2, \pm_{2}}))}_{X^{s,b-1+\eps_{0}}_{\pm_{0}}}
	\aleq & \nrm{\phi_{1, \pm_{1}}}_{X^{s+1/2,b}_{\pm_{1}}} \nrm{\phi_{2, \pm_{2}}}_{X^{s+1/2,b}_{\pm_{2}}}. \label{eq:CSH:est4N:1:2}
\end{align}
By duality and self-adjointness of $\mR^{\nu}_{\pm_{0}}$, \eqref{eq:CSH:est4N:1:1} follows from
\begin{equation} \label{eq:CSH:nf4N}
\begin{aligned}
	& \abs{\int \eps_{\mu \nu \lmb} \Im(\overline{\phi_{1, \pm_{1}}} \mR^{\lmb}_{\pm_{2}} \phi_{2,\pm_{2}}) \mR^{\nu}_{\pm_{0}} \phi_{0, \pm_{0}} \, \ud t \ud x} \\
	&\quad \aleq \nrm{\phi_{0, \pm_{0}}}_{X^{-s, 1-b+\eps_{0}}_{\pm_{0}}}\nrm{\phi_{1, \pm_{1}}}_{X^{s+1/2,b}_{\pm_{1}}} \nrm{\phi_{2, \pm_{2}}}_{X^{s-1/2,b}_{\pm_{2}}}
\end{aligned}
\end{equation}
for all combinations of signs and Schwartz $\phi_{1, \pm_{1}}, \phi_{2, \pm_{2}}, \phi_{3, \pm_{3}}$. As in \eqref{eq:CSD:nf4N2}, the $\nf^{\lmb \nu}_{\pm_{2}, \pm_{0}}$ null form between $\phi_{2}, \phi_{0}$ is evident.

\subsubsection*{Wave equation for $\phi$}
 Decomposing $A_{\mu}$ into half-waves and applying the formula \eqref{eq:CSH:rd2mR} to $\phi$, \eqref{eq:CSH:est4M:1} is reduced to showing
\begin{align}
\nrm{A^{\mu}_{\pm_{1}} \mR_{\pm_{2}, \mu} \abs{\nb} \phi_{\pm_{2}}}_{X^{s-1/2,b-1+\eps_{0}}_{\pm_{0}}} 
	\aleq &  \nrm{\bfA_{\pm_{1}}}_{X^{s, b}_{\pm_{1}}} \nrm{\phi_{\pm_{2}}}_{X^{s+1/2,b}_{\pm_{2}}}  \label{eq:CSH:est4M:1:1}, \\
\nrm{A^{\mu}_{\pm_{1}} \err_{\mu} \phi_{\pm_{2}}}_{X^{s-1/2,b-1+\eps_{0}}_{\pm_{0}}} 
	\aleq & \nrm{\bfA_{\pm_{1}}}_{X^{s, b}_{\pm_{1}}} \nrm{\phi_{\pm_{2}}}_{X^{s+1/2,b}_{\pm_{2}}}  \label{eq:CSH:est4M:1:2}
\end{align}
for all combinations of signs $(\pm_{0}, \pm_{1}, \pm_{2})$ and Schwartz $A_{\nu, \pm_{1}}, \phi_{\pm_{2}}$. Proceeding as in the last part of \S \ref{subsec:CSD:nullform}, \eqref{eq:CSH:est4M:1:1} follows once we prove the null form estimates
\begin{align} 
	\nrm{\nf^{12}_{\pm_{1}, \pm_{2}}(B_{\pm_{1}}, \phi_{\pm_{2}})}_{X^{s-1/2, b-1+\eps_{0}}_{\pm_{0}}} 
	\aleq &  \nrm{B_{\pm_{1}}}_{X^{s,b}_{\pm_{1}}} \nrm{\phi_{\pm_{2}}}_{X^{s-1/2,b}_{\pm_{2}}}   \label{eq:CSH:nf4M:1}, \\
	\nrm{\nf^{0}_{\pm_{1}, \pm_{2}} (A_{0, \pm_{1}}, \phi_{\pm_{2}}) }_{X^{s-1/2, b-1+\eps_{0}}_{\pm_{0}}} 
	\aleq &  \nrm{A_{0, \pm_{1}}}_{X^{s,b}_{\pm_{1}}} \nrm{\phi_{\pm_{2}}}_{X^{s-1/2,b}_{\pm_{2}}}  \label{eq:CSH:nf4M:2}
\end{align}
for all combinations of signs $(\pm_{0}, \pm_{1}, \pm_{2})$ and Schwartz inputs.
\subsection{Proof of the global estimates}
In this subsection, we will finally find conditions on $(s, b)$ required to establish \eqref{eq:CSH:est4N:1}--\eqref{eq:CSH:est4M:2}.

\subsubsection*{Linear estimate}
Let us establish the linear estimate \eqref{eq:CSH:est4M:0}. Given that
\begin{equation} \label{eq:CSH:cond4sb:1}
	0 < b < 1,
\end{equation}
and $\eps_{0}$ is small enough, we have for each sign $\pm$
\begin{equation*}
	\nrm{\phi}_{X^{s-1/2, b-1+\eps_{0}}_{\pm}} \leq \nrm{\brk{\nb}^{s-1/2} \phi}_{L^{2}_{t,x}} \leq \sum_{\pm_{1}} \nrm{\phi_{\pm_{1}}}_{X^{s+1/2,b}_{\pm_{1}}}.
\end{equation*}

\subsubsection*{Bilinear estimates}
Here we shall prove \eqref{eq:CSH:est4N:1}, \eqref{eq:CSH:est4M:1}, which have been reduced in \S \ref{subsec:CSH:nullform} to \eqref{eq:CSH:est4N:1:2}, \eqref{eq:CSH:nf4N}, \eqref{eq:CSH:est4M:1:2}, \eqref{eq:CSH:nf4M:1} and \eqref{eq:CSH:nf4M:2}. As $X^{s,b}_{\pm} \subset H^{s,b} \subset H^{s,-b} \subset X^{s,-b}_{\pm}$ for $s \in \bbR$ and $b \geq 0$, it suffices to establish the analogous estimates with $X^{s,b}_{\pm}$ replaced by $H^{s,b}$, provided that \eqref{eq:CSH:cond4sb:1} holds and $\eps_{0} > 0$ is sufficiently small.

Let us first get the estimates \eqref{eq:CSH:est4N:1:2}, \eqref{eq:CSH:est4M:1:2} (which involve the error operator $\calE_{\mu}$) out of the way. As all the norms involved depend only on the size of the space-time Fourier transform, we may throw away $\mR^{\nu}_{\pm_{0}}$, $\err_{\mu}$; then it suffices to establish
\begin{equation} \label{eq:CSH:est4err}
	\left\{
	\begin{aligned}
		&\nrm{\phi_{1} \phi_{2}}_{H^{s, 1-b}} \aleq \nrm{\phi_{1}}_{H^{s+1/2, b}} \nrm{\phi_{2}}_{H^{s+1/2, b}}, \\
		&\nrm{\phi_{1} \phi_{2}}_{H^{s-1/2, 1-b}} \aleq \nrm{\phi_{1}}_{H^{s, b}} \nrm{\phi_{2}}_{H^{s+1/2, b}} 
	\end{aligned}
	\right.
\end{equation}
for Schwartz $\phi_{1}, \phi_{2}$. Assuming
\begin{equation} \label{eq:CSH:cond4sb:2}
	1/2 < b < 1, \quad s \geq 0,
\end{equation}
and $\eps_{0}$ is sufficiently small, we can readily verify that \eqref{eq:CSH:est4err} holds by Theorem \ref{thm:HsbProd}.

Next, let us turn to the null form estimates \eqref{eq:CSH:nf4N}, \eqref{eq:CSH:nf4M:1} and \eqref{eq:CSH:nf4M:2}. Applying Lemma \ref{lem:nullform}, we see that all of these estimates would follow if we prove
\begin{align} 
	& \nrm{\mathfrak{B}^{1}_{\pm_{1}, \pm_{2}}(\phi_{1}, \phi_{2})}_{H^{s-1/2, b-1+\eps_{0}}} \aleq \nrm{\phi_{1}}_{H^{s,b}} \nrm{\phi_{2}}_{H^{s-1/2, b}} \label{eq:CSH:nfEst:1}, \\
	& \nrm{\mathfrak{B}^{1}_{\pm_{1}, \pm_{2}}(\phi_{1}, \phi_{2})}_{H^{-s-1/2,-b}} \aleq \nrm{\phi_{1}}_{H^{-s,1-b+\eps_{0}}} \nrm{\phi_{2}}_{H^{s-1/2,b}} \label{eq:CSH:nfEst:2}
\end{align}
for Schwartz $\phi_{1}, \phi_{2}$. By Lemma \ref{lem:est4angle}, to prove \eqref{eq:CSH:nfEst:1}, it suffices to show
\begin{equation} \label{eq:CSH:nfEst:1:1}
\left\{
\begin{aligned}
	& \nrm{\phi_{1} \phi_{2}}_{H^{s-1/2,b-1/2+\eps_{0}}} \aleq \nrm{\phi_{1}}_{H^{s+1/2,b}} \nrm{\phi_{2}}_{H^{s-1/2,b}}, \\
	& \nrm{\phi_{1} \phi_{2}}_{H^{s-1/2,b-1/2+\eps_{0}}} \aleq \nrm{\phi_{1}}_{H^{s,b}} \nrm{\phi_{2}}_{H^{s,b}}, \\
	& \nrm{\phi_{1} \phi_{2}}_{H^{s-1/2,b-1+\eps_{0}}} \aleq \nrm{\phi_{1}}_{H^{s+1/2,b-1/2}} \nrm{\phi_{2}}_{H^{s-1/2,b}}, \\
	& \nrm{\phi_{1} \phi_{2}}_{H^{s-1/2,b-1+\eps_{0}}} \aleq \nrm{\phi_{1}}_{H^{s,b-1/2}} \nrm{\phi_{2}}_{H^{s,b}}, \\
	& \nrm{\phi_{1} \phi_{2}}_{H^{s-1/2,b-1+\eps_{0}}} \aleq \nrm{\phi_{1}}_{H^{s+1/2,b}} \nrm{\phi_{2}}_{H^{s-1/2,b-1/2}},
\end{aligned}
\right.
\end{equation}
and for \eqref{eq:CSH:nfEst:2}, it is enough to prove
\begin{equation} \label{eq:CSH:nfEst:2:1}
\left\{
\begin{aligned}
	& \nrm{\phi_{1} \phi_{2}}_{H^{-s-1/2,-b+1/2}} \aleq \nrm{\phi_{1}}_{H^{-s+1/2,1-b+\eps_{0}}} \nrm{\phi_{2}}_{H^{s-1/2,b}}, \\
	& \nrm{\phi_{1} \phi_{2}}_{H^{-s-1/2,-b+1/2}} \aleq \nrm{\phi_{1}}_{H^{-s,1-b+\eps_{0}}} \nrm{\phi_{2}}_{H^{s,b}}, \\
	& \nrm{\phi_{1} \phi_{2}}_{H^{-s-1/2,-b}} \aleq \nrm{\phi_{1}}_{H^{-s+1/2,1/2-b+\eps_{0}}} \nrm{\phi_{2}}_{H^{s-1/2,b}}, \\
	& \nrm{\phi_{1} \phi_{2}}_{H^{-s-1/2,-b}} \aleq \nrm{\phi_{1}}_{H^{-s,1/2-b+\eps_{0}}} \nrm{\phi_{2}}_{H^{s,b}}, \\
	& \nrm{\phi_{1} \phi_{2}}_{H^{-s-1/2,-b}} \aleq \nrm{\phi_{1}}_{H^{-s+1/2,1-b+\eps_{0}}} \nrm{\phi_{2}}_{H^{s-1/2,b-1/2}}, \\
	& \nrm{\phi_{1} \phi_{2}}_{H^{-s-1/2,-b}} \aleq \nrm{\phi_{1}}_{H^{-s,1-b+\eps_{0}}} \nrm{\phi_{2}}_{H^{s,b-1/2}}.
\end{aligned}
\right.
\end{equation}
By Theorem \ref{thm:HsbProd}, we can verify that \eqref{eq:CSH:nfEst:1:1} and \eqref{eq:CSH:nfEst:2:1} hold provided that
\begin{equation} \label{eq:CSH:cond4sb:3}
	1/2 < b < 1, \qquad s > \max \bb\{ \frac{1}{4},\, b - \frac{1}{2},\, 1-b, \,\frac{b}{3} \bb\},
\end{equation}
and $\eps_{0}$ is sufficiently small.

\subsubsection*{Cubic estimates}
In order to prove the cubic estimates \eqref{eq:CSH:est4N:2} and \eqref{eq:CSH:est4M:2}, it is enough to show that (again throwing away $\mR^{\nu}_{\pm}$)
\begin{align} 
	& \nrm{\phi_{1} \phi_{2} \phi_{3}}_{H^{s,b-1+\eps_{0}}} \aleq \nrm{\phi_{1}}_{H^{s+1/2,b}} \nrm{\phi_{2}}_{H^{s+1/2,b}} \nrm{\phi_{3}}_{H^{s,b}} \label{eq:CSH:cubicEst:1}, \\
	& \nrm{\phi_{1} \phi_{2} \phi_{3}}_{H^{s-1/2,b-1+\eps_{0}}} \aleq \nrm{\phi_{1}}_{H^{s+1/2,b}} \nrm{\phi_{2}}_{H^{s,b}} \nrm{\phi_{3}}_{H^{s,b}} \label{eq:CSH:cubicEst:2}
\end{align}
hold for Schwartz $\phi_{1}, \phi_{2}, \phi_{3}$.
Assuming that \eqref{eq:CSH:cond4sb:3} holds and $\eps_{0}$ is small enough, then by Theorem \ref{thm:HsbProd}, we prove \eqref{eq:CSH:cubicEst:1} and \eqref{eq:CSH:cubicEst:2} as follows.
\begin{align*}
	&\nrm{\phi_{1} \phi_{2} \phi_{3}}_{H^{s,b-1+\eps_{0}}} 
	\aleq \nrm{\phi_{1}}_{H^{s+1/2,b}} \nrm{\phi_{2} \phi_{3}}_{H^{s,0}} 
	\aleq \nrm{\phi_{1}}_{H^{s+1/2,b}} \nrm{\phi_{2}}_{H^{s+1/2,b}} \nrm{\phi_{3}}_{H^{s,b}}, \\ 
	&\nrm{\phi_{1} \phi_{2} \phi_{3}}_{H^{s-1/2,b-1+\eps_{0}}} 
	\aleq \nrm{\phi_{1} \phi_{2}}_{H^{s, 0}} \nrm{\phi_{3}}_{H^{s,b}} 
	\aleq \nrm{\phi_{1}}_{H^{s,b}} \nrm{\phi_{2}}_{H^{s+1/2,b}} \nrm{\phi_{3}}_{H^{s,b}}. 
\end{align*}

\subsubsection*{Conclusion}
Observe that for $s = 1/4+\eps$ and $b = 3/4-2\eps$ for an arbitrary $\eps$ satisfying $0 < \eps \ll 1$, and choosing $\eps_{0}$ small enough, the conditions \eqref{eq:CSH:cond4sb:1}, \eqref{eq:CSH:cond4sb:2} and \eqref{eq:CSH:cond4sb:3} are satisfied. This proves Theorem \ref{thm:lwp4CSH}.
\\
\ \\
{\bf Acknowledgements.}\\
H. Huh was supported by Basic Science Research Program
through the National Research Foundation of Korea(NRF) funded by
the Ministry of Education, Science and Technology (2011-0015866) and also partially supported
by the TJ Park Junior Faculty Fellowship. 
S.-J. Oh would like to thank KIAS for hospitality where a part of the research was conducted. S.-J. Oh was supported by the Samsung Scholarship.

\bibliographystyle{amsplain}

\end{document}